\documentclass[12pt,a4paper,reqno]{amsart}
\usepackage{amssymb}
\usepackage{amscd}
\usepackage{enumerate}
\numberwithin{equation}{section}

 \usepackage{mathtools}

\theoremstyle{plain}

\newtheorem{theorem}{Theorem}[section]
\newtheorem{proposition}[theorem]{Proposition}
\newtheorem{lemma}[theorem]{Lemma}
\newtheorem{corollary}[theorem]{Corollary}
\newtheorem{conjecture}[theorem]{Conjecture}

\theoremstyle{definition}

\theoremstyle{remark}
\newtheorem{remark}[theorem]{Remark}

\newcommand\R{\mathbb{R}}

\newcommand\C{\mathbb{C}}

\def\stacksum#1#2{{\stackrel{{\scriptstyle #1}}
{{\scriptstyle #2}}}}

\newcommand{\sump}{\sideset{}{^\prime}\sum}

\usepackage{hyperref}

\usepackage{algorithm, algorithmic}

\begin{document}

\title{Primes in intervals of bounded length}

\author{Andrew Granville}
\address{D\'epartement de math\'ematiques et de statistiques,  Universit\'e de Montr\'eal, Montr\'eal QC H3C 3J7, Canada. }
\email{andrew@dms.umontreal.ca}
\thanks{To Yitang Zhang, for showing that one can, no matter what}
\thanks{Thanks to Sary Drappeau, Tristan Freiberg, John Friedlander, Adam Harper, Dimitris Koukoulopoulos, Emmanuel Kowalski, James Maynard, Pace Nielsen, Robert Lemke Oliver,  Terry Tao and the anonymous referee for their comments on earlier versions of this article.}

\subjclass{11P32}

\begin{abstract}  The infamous \textsl{Twin Prime conjecture} states that there are infinitely many   pairs of distinct primes which differ by $2$.  Until recently this conjecture had seemed to be far out of reach with current techniques. However, in April 2013, Yitang Zhang proved the existence of a finite bound $B$ such that there are infinitely many pairs of distinct primes which differ by no more than $B$.  This is a massive breakthrough, making the twin prime conjecture look highly plausible, and the techniques developed help us to better understand other delicate questions about prime numbers that had previously seemed intractable.

Zhang even showed that one can take   $B =  70000000$.  Moreover, a co-operative team, \emph{polymath8}, collaborating only on-line, had been able to lower the value of $B$ to ${4680}$. They had not only been more careful in several difficult arguments in Zhang's original paper, they had also developed Zhang's techniques to be both more powerful and to allow a much simpler proof (and forms the basis for the proof presented herein).

In November 2013, inspired by Zhang's extraordinary breakthrough, James Maynard dramatically slashed this bound to $600$, by a substantially easier method. Both Maynard, and Terry Tao who had independently developed the same idea, were able to extend their proofs to show that for any given integer $m\geq 1$ there exists a bound $B_m$ such that there are infinitely many intervals of length $B_m$ containing at least $m$ distinct primes. We will also prove this much stronger result herein, even showing that one can take $B_m=e^{8m+5}$.

If Zhang's method is combined with the Maynard-Tao set up then it appears that the bound can be further reduced to $246$.
If all of these techniques could be pushed to their limit then we would obtain  $B$($=B_2$)$=12$ (or arguably to $6$), so new ideas are still needed to have a feasible plan for proving the twin prime conjecture.

The article will be split into two parts. The first half  will introduce the work of Zhang, Polymath8, Maynard and Tao, and explain their arguments that allow them to prove their spectacular results. The second half of this article develops a proof of Zhang's main novel contribution, an estimate for primes in relatively short arithmetic progressions.
\end{abstract}

\maketitle

\setcounter{tocdepth}{1}

\newpage
\part{Primes in short intervals}

\section{Introduction}\label{sec:intro}

\subsection{Intriguing questions about primes}

Early on in our mathematical education we get used to the two basic rules of arithmetic, addition and multiplication. Then prime numbers are defined, not in terms of what they are, but rather  in terms of what they are not (i.e. that they \emph{cannot} be factored into two smaller integers)). This makes them difficult to find, and to work with.

  Prime numbers can be seen to occur rather frequently:
\[
2, 3, 5, 7, 11, 13, 17, 19, 23, 29, 31, 37, 41, 43, 47, 53, 59, 61,\ldots
\]
but it took a rather clever construction of the ancient Greeks to establish that there really are infinitely many.  Looking closely, patterns begin to emerge; for example, primes often come in pairs:
\[
3 \ \text{and} \ 5;\ 5 \ \text{and} \  7;\  11 \ \text{and} \ 13; \ 17 \ \text{and} \  19;\  29\ \text{and} \ 31;\ 41 \ \text{and} \ 43;\ 59 \ \text{and} \  61,\ldots
\]
One might guess that there are infinitely many such prime pairs. But this is an open, elusive question, the \emph{twin prime conjecture}. Until recently there was little theoretical evidence for it. A lot of data suggesting that these pairs never quit; and the higher view that it fits like the central piece  of an enormous jigsaw of \emph{conjectures} of all sorts of prime patterns. If the twin prime conjecture were false then one would have to be sceptical of all these conjectures, and our intellectual world would be the poorer for it.

The twin prime conjecture is  intriguing to amateur and professional mathematicians alike. It asks for a very delicate additive property of a sequence defined by its multiplicative properties, which some might argue makes it an artificial question. Indeed, number theorists had struggled  to identify an approach to this question that captured its essence enough to allow us to make headway.   But recently an approach has been found that puts the question firmly within the framework of \textsl{sieve theory} which has allowed the proof of important steps towards the eventual resolution of the twin prime conjecture (and its generalizations).

In the first few sections we  take a leisurely stroll through the historical and mathematical background, so as to give the reader a sense of the great theorems that have been recently proved, from a perspective that will prepare the reader for the details of the proof.

\subsection{Other patterns}  Staring at the list of primes above we
 find four primes which have all the same digits, except the last one:
\[
11, 13, 17 \  \ \text{and} \ 19; \text{ which is repeated with } 101, 103, 107, 109; \text{ then } 191, 193, 197, 199
\]
and one can find many more such examples -- are there infinitely many?  More simply how about prime pairs with difference $4$:
\[
3 \ \text{and} \ 7;\ 7 \ \text{and} \  11;\  13 \ \text{and} \ 17;\ 19 \ \text{and} \  23;\  37\ \text{and} \ 41;\ 43 \ \text{and} \ 47;\ 67 \ \text{and} \  71,\ldots;
\]
or difference $10$:
\[
3 \ \text{and} \ 13;\ 7 \ \text{and} \  17;\  13 \ \text{and} \ 23;\ 19 \ \text{and} \  29;\  31\ \text{and} \ 41;\ 37 \ \text{and} \ 47;\ 43 \ \text{and} \  53,\ldots ?
\]
Are there infinitely many such pairs? Such questions were probably asked back to antiquity, but the first clear mention of twin primes in the literature  appears in a presentation by Alphonse de Polignac, a student at the \'Ecole Polytechnique in Paris, in 1849. In his honour we now call any integer $h$, for which there are infinitely many prime pairs $p, p+h$, a \emph{de Polignac number}.\footnote{De Polignac also required that   $p$ and $p+h$    be consecutive primes, though this requirement is not essential to our discussion here. De Polignac's article \cite{deP} is very much that of an amateur mathematician, developing a first understanding of the sieve of Eratosthenes. His other ``conjecture'' in the paper,  asking whether every odd number is the sum of a prime and power of two, is false for as small an example as $127$.}

Then there are the \emph{Sophie Germain pairs}, primes $p$ and $q:=2p+1$, which prove useful in several simple algebraic constructions:\footnote{The group of reduced residues mod $q$ is a cyclic group of order $q-1=2p$, and therefore isomorphic to $C_2\times C_p$ if $p>2$. Hence the order of each element in the group is either $1$ (that is, $1\pmod q$), $2$ (that is, $-1\pmod q$), $p$ (the squares mod $q$) or $2p=q-1$. Hence $g \pmod q$ generates the group of reduced residues if and only if $g$ is not a square mod $q$ and $g\not\equiv -1\pmod q$.}
\[
2 \ \text{and} \ 5;\ 3 \ \text{and} \  7;\  5 \ \text{and} \ 11;\ 11 \ \text{and} \  23;\  23\ \text{and} \ 47;\ 29 \ \text{and} \ 59;\ 41 \ \text{and} \  83; \ldots;
\]

Can one predict which prime patterns can occur and which do not? Let's start with differences between primes: One of any two consecutive integers must be even, and so can be prime only if it equals $2$. Hence  there is just the one pair, 2 and 3, of primes with difference $1$. One can make a similar argument for prime pairs with odd difference. Hence if $h$ is an integer for which there are infinitely many prime pairs of the form $p,\ q=p+h$ then $h$ must be even. We discussed examples for $h=2$, for $h=4$ and for $h=10$ above, and the reader can similarly construct lists of examples for $h=6$ and for $h=8$, and indeed for any other even $h$ that takes her or his fancy. This leads us to bet on the   \emph{generalized twin prime conjecture}, which states that for any even integer $2k$ there are infinitely many prime pairs $p,\ q=p+2k$.

What about prime triples? or quadruples? We saw two examples of prime quadruples of the form $10n+1,\ 10n+3,\ 10n+7,\ 10n+9$, and believe that there are infinitely many. What about other patterns?  Evidently any pattern that includes an odd difference cannot succeed. Are there any other obstructions?  The simplest pattern that avoids an odd difference is $n, n+2, n+4$. One finds the one example $3, \ 5,\ 7$ of such a prime triple, but no others. Further examination makes it clear why not: One of the three numbers is always divisible by 3.  This is analogous to one of $n, n+1$ being divisible by $2$; and, similarly, one of
$n, n+6, n+12, n+18, n+24$ is always divisible by $5$. The general obstruction can be described as follows:

For a given set of distinct integers $a_1<a_2<\ldots<a_k$ we say that prime $p$ is an \emph{obstruction} if $p$ divides at least one of $n+a_1,\ldots, n+a_k$, for every integer $n$. In other words, $p$ divides
\[
\mathcal P(n)=(n+a_1)(n+a_2)\ldots (n+a_k)
\]
for every integer $n$; which can be classified by the condition that the set $a_1, a_2,\ldots,a_k \pmod p$ includes all of the residue classes mod $p$.  If no prime is an obstruction then we say that $x+a_1,\ldots, x+a_k$ is an \emph{admissible} set of forms.\footnote{Notice that $a_1, a_2,\ldots,a_k \pmod p$ can occupy no more than $k$ residue classes mod $p$ and so, if $p>k$ then $p$ cannot be an obstruction. Hence, to check whether a given set $A$ of $k$ integers is admissible, one needs only find one residue class $b_p \pmod p$, for each prime $p\leq k$, which does not contain any element of $A$.}.

In 1904 Dickson made the optimistic conjecture that if there is no such obstruction to a set of linear forms being infinitely often prime, then they are infinitely often simultaneously prime. That is:
\smallskip

\textbf{Conjecture}:\ \emph{ If $x+a_1,\ldots, x+a_k$ is an admissible  set of forms  then there are infinitely many integers $n$ such that $n+a_1,\ldots, n+a_k$ are all prime numbers.}
\smallskip

In this case, we call $n+a_1,\ldots, n+a_k$ a \emph{$k$-tuple} of prime numbers. 
Dickson's \textsl{prime $k$-tuple conjecture} states that if a set
$b_1x+a_1,\ldots, b_kx+a_k$ of linear forms is admissible 
(that is, if the forms are all positive at infinitely many integers $x$ and, for each prime $p$ there exist an integer 
$n$ such that $p\nmid \mathcal P(n):=\prod_j (b_jn+a_j)$)   then there are infinitely many integers $n$ for which $b_1n+a_1,\ldots, b_kn+a_k$ are all primes.

To date, this has not been proven for any $k>1$ though, following Zhang's work, we begin to get close for $k=2$. Indeed, Zhang has proved a weak variant of this conjecture for $k=2$, as we shall see. Moreover Maynard \cite{maynard}, and Tao \cite{tao}, have gone on to prove a weak variant \emph{for any} $k\geq 2$.

The above conjecture can be extended to linear forms in more than one variable (for example the set of forms $m, m+n, m+4n$):

\textbf{The prime $k$-tuplets conjecture}:\ \emph{ If a set of $k$ linear forms in $n$ variables is admissible then there are infinitely many sets of $n$ integers such  that when we substitute these integers into the forms we get a $k$-tuple of prime numbers.}
\smallskip

There has been substantial recent progress on this conjecture.
The famous breakthrough was Green and Tao's theorem \cite{GT} for the $k$-tuple of linear forms in the two variables $a$ and $d$:\
\[
a,\ a+d,\ a+2d,\ldots, \ a+(k-1)d
\]
(in other words, there are infinitely many $k$-term arithmetic progressions of primes.)
Along with Ziegler, they went on to prove the prime $k$-tuplets conjecture for \emph{any} admissible set of linear forms, provided no two satisfy a linear equation over the integers, \cite{GTZ}. What a remarkable theorem! Unfortunately these exceptions include many of the questions we are most interested in; for example, $p,\ q=p+2$ satisfy the linear equation $q-p=2$; and  $p,\ q=2p+1$ satisfy the linear equation $q-2p=1$).

Finally,  we also believe that the conjecture holds if we consider any admissible set of $k$ irreducible polynomials with integer coefficients, with any number of variables. For example we believe that $n^2+1$ is infinitely often prime, and that there are infinitely many prime triples $m,\ n,\  m^2-2n^2$.

\subsection{The new results; primes in bounded intervals}

In this section we state Zhang's main theorem, as well as the improvement of Maynard and Tao,  and discuss a few of the more beguiling consequences:
\smallskip

\textbf{Zhang's main theorem}:\ \emph{There exists an integer $k$ such that if $x+a_1,\ldots, x+a_k$ is an admissible  set of forms  then there are infinitely many integers $n$ for which} at least two of \emph{$n+a_1,\ldots, n+a_k$ are prime numbers.}
\smallskip

Note that the result states that only two of the $n+a_i$ are prime, not all (as would be required in the prime $k$-tuplets conjecture). Zhang proved this result for a fairly large value of $k$, that is $k=3500000$, which has been reduced to $k=105$ by Maynard, and now to $k=50$ in \cite{polymath8b}.  Of course if one could take $k=2$ then we would have the twin prime conjecture,\footnote{And the generalized twin prime conjecture, and that there are infinitely many Sophie Germain pairs (if one could use non-monic polynomials), and $\ldots$} but the most optimistic plan at the moment, along the lines of Zhang's proof, would yield $k=3$ (see section \ref{genseqs}).

To deduce that there are bounded gaps between primes from Zhang's Theorem we need only show the existence of an admissible set with $k$ elements. This is not difficult, simply by letting the $a_i$ be the first $k$ primes $>k$.\footnote{This is admissible since none of the $a_i$ is $0 \pmod p$ for any $p\leq k$, and the $p>k$ were handled in the previous footnote.} Hence we have proved:

\begin{corollary}   [Bounded gaps between primes]  There exists a  bound $B$ such that there are infinitely many integers pairs of prime numbers $p<q<p+B$.
\end{corollary}

Finding the narrowest admissible $k$-tuples is a challenging question.  The prime number theorem together with our construction above implies that $B\leq k(\log k+C)$ for some constant $C$, but it is interesting to get better bounds. For Maynard's $k=105$, Engelsma exhibited an admissible $105$-tuple of width $600$, and proved that there are no narrower ones. The narrowest $50$-tuple has width $246$ and one such tuple is: \smallskip


\centerline{$0, 4, 6, 16, 30, 34, 36, 46, 48, 58, 60, 64, 70, 78, 84, 88, 90, 94, 100, 106, $}

\centerline{$108, 114, 118, 126, 130, 136, 144, 148, 150, 156, 160, 168, 174, 178, 184,  $}

\centerline{$ 190, 196, 198, 204, 210, 214,  216, 220, 226, 228, 234, 238, 240, 244, 246.$\footnote{Sutherland's website {\tt  http://math.mit.edu/$\sim$primegaps/} lists narrowest $k$-tuples for all small $k$.}}

The Corollary further implies (for $B=246$)

\begin{corollary}   There is an integer $h, 0< h\leq B$ such that there are infinitely many pairs of primes $p, p+h$.
\end{corollary}

That is, some positive integer $\leq B$ is a de Polignac number.
In fact one can go a little further using Zhang's main theorem, and deduce that if $A$ is \emph{any} admissible set of $k$ integers then there is an integer $h\in (A-A)^+:=\{ a-b:\ a>b \in A\}$ such that there are infinitely many pairs of primes $p, p+h$. One can find many beautiful consequences of this; for example, that a positive proportion of even integers are de Polignac numbers.

Zhang's theorem can be proved for $k$-tuplets $b_1x+a_1,\ldots,b_kx+a_k$ with minor (and obvious) modifications to the proof given herein.

Next we state the Theorem of Maynard and of Tao:

\textbf{The Maynard-Tao theorem}:\ \emph{For any given integer $m\geq 2$,  there exists an integer $k$ such that if $x+a_1,\ldots, x+a_k$ is an admissible  set of forms  then there are infinitely many integers $n$ for which} at least \emph{$m$ of $n+a_1,\ldots, n+a_k$ are prime numbers.}
\smallskip

This includes and extends Zhang's Theorem (which is the case $k=2$).  The proof even allows one make this explicit (we will obtain $k\leq e^{8m+4}$, and Maynard improves this to $k\leq c m^2 e^{4m}$ for some constant $c>0$).

\begin{corollary}  [Bounded intervals with $m$ primes]  For any given integer $m\geq 2$, there exists a  bound $B_m$ such that there are infinitely many intervals $[x,x+B_m]$ (with $x\in \mathbb Z$) which contain $m$ prime numbers.
\end{corollary}

We will prove that one can take $B_m=e^{8m+5}$ (which Maynard improves  to $B_m=cm^3e^{4m}$, and the polymath team \cite{polymath8b} to $B_m=cme^{(4-\frac{28}{157})m}$, for some constant $c>0$).

A \emph{Dickson $k$-tuple} is a set of integers $a_1<\ldots <a_k$ such that there are infinitely many integers for which $n+a_1, n+a_2,\ldots, n+a_k$ are each prime.

\begin{corollary}  A positive proportion of $m$-tuples of integers are Dickson $m$-tuples.
\end{corollary}

\begin{proof}  With the notation as in the Maynard-Tao theorem let $R=\prod_{p\leq k} p$, select $x$ to be a large integer multiple of $R$ and let $\mathcal N:=\{ n\leq x:\ (n,R)=1\}$ so that
$|\mathcal N| =\frac{\phi(R)}R x$. Any subset of $k$ elements of $\mathcal N$ is admissible, since it does not contain any integer $\equiv 0 \pmod p$ for each prime $p\leq k$.  There are $\binom{ |\mathcal N|}{k}$ such $k$-tuples. Each contains a Dickson $m$-tuple by the Maynard-Tao theorem.

Now suppose that are $T(x)$ Dickson $m$-tuples that are subsets of $\mathcal N$. Any such $m$-tuple is a subset of exactly $\binom{ |\mathcal N|-m}{k-m}$ of the $k$-subsets of $\mathcal N$, and hence
\[
T(x) \cdot \binom{ |\mathcal N|-m}{k-m} \geq \binom{ |\mathcal N|}{k},
\]
and therefore $T(x)\geq (|\mathcal N|/k)^m=(\frac{\phi(R)}R/k)^m \cdot x^m$ as desired. \end{proof}

This proof yields that, as a proportion of the $m$-tuples in $\mathcal N$,
\[
T(x) \big/ \binom{ |\mathcal N|}{m} \geq 1 \big/ \binom{k}{m}  .
\]
The $m=2$ case implies that at least $\frac 1{5460}$th of the even integers are de Polignac numbers. (This is improved to at least $\frac 1{181}$ in \cite{GKKO}, which also discusses limitations on what can be deduced from a result like Zhang's Theorem.)

Zhang's Theorem and the Maynard-Tao theorem each hold for any admissible $k$-tuple of linear forms (not just those of the form $x+a$). With this we can prove several other amusing consequences:

$\bullet$\ The last Corollary holds if we insist that the primes in the Dickson $k$-tuples are consecutive primes.

$\bullet$ \ There exists a constant $H$ such that every interval $[x,x+H]$ contains a de Polignac number (see \cite{pintz-polignac}).

$\bullet$\  There are infinitely many $m$-tuples of consecutive primes such that each pair in the $m$-tuple differ from one another by just two digits when written in base $10$.

 $\bullet$\  For any $m\geq 2$ and coprime integers $a$ and $q$, there are infinitely many intervals $[x,x+qB_m]$ (with $x\in \mathbb Z$) which contain exactly $m$ prime numbers, each $\equiv a \pmod q$.\footnote{Thanks to Tristan Freiberg for pointing this out to me (see also \cite{freiberg}). However, I do not see how to modify the proof to show, given $r_1,\ldots, r_m$ coprime to $q$, that one has primes $p_{n+1},\ldots,p_{n+m}$ with $p_{n+j}\equiv r_j \pmod q$ for $j=1,\ldots,m$}

 $\bullet$ For any   integer $r\geq 2$ there are infinitely many $m$-tuples of distinct primes
 $q_1,\ldots,q_m$, such that the ratios $(q_i-1)/(q_j-1)$ are all (bounded) powers of $r$.

 $\bullet$\ Let $d_n=p_{n+1}-p_n$ where $p_n$ is the $n$th smallest prime.  Fix $m\geq 1$. There are infinitely many $n$ for which $d_n<d_{n+1}<\ldots < d_{n+m}$.  There are also infinitely many $n$ for which $d_n>d_{n+1}>\ldots >d_{n+m}$. (See \cite{consecutive}.) This was a favourite problem of Paul Erd\H os, though we do not see how to deduce such a result for other orderings of the $d_n$.\footnote{It was also shown in \cite{consecutive} that the $d_{n+j}$ can grow as fast as one likes. Moreover that one can insist that  $d_n|d_{n+1}|\ldots | d_{n+m}$ }

$\bullet$ One can also deduce \cite{lola} that there are infinitely many $n$ such that $s_r(p_n)<s_r(p_{n+1})<\ldots < s_r(p_{n+m})$, where $s_r(N)$ denotes the sum of the digits of $N$ when written in base $r$ (as well as
$s_r(p_n)> \ldots > s_r(p_{n+m})$).

$\bullet$ Moreover \cite{lola}  there are infinitely many $n$ such that $\phi(p_n-1)<\phi(p_{n+1}-1)<\ldots < \phi(p_{n+m}-1)$,  (as well as $\phi(p_n-1)> \ldots > \phi(p_{n+m}-1)$). An analogous result holds with $\phi$ replaced by $\sigma, \tau, \nu$ and many other arithmetic functions.

$\bullet$ If $\alpha$ is an algebraic, irrational number then \cite{chua} there are infinitely $n$ such that at least $m$ of
$[\alpha n], [\alpha(n+1)],\ldots , [\alpha(n+k)]$ are prime (where $[t]$ denotes the integer part of $t$). This result can be extended to any   irrational number $\alpha$ for which there exists $r$ such that $|p\alpha -q|\geq 1/p^r$ for all integers $p,q>0$.

In the eight months since Maynard's preprint,   many further interesting applications of the technique that have appeared, some of which we discuss in section \ref{FurtherApples}.



\subsection{Bounding the gaps between primes. A brief history.}
The young Gauss,  examining Chernac's table of primes up to one million, guessed that ``the density of primes at around $x$ is roughly $1/\log x$''. This was subsequently shown to be, as a consequence of the \emph{prime number theorem}.  Therefore we are guaranteed that there are infinitely many pairs of primes $p<q$ for which $q-p\leq (1+\epsilon)\log p$ for any fixed $\epsilon>0$, which is not quite as small a gap as we are hoping for! Nonetheless this raises the question: Fix $c>0$. Can we even prove that

\begin{center}
\emph{There are infinitely many pairs of primes $p<q$ with $q<p+c\log p$} ?
\end{center}

This follows for all $c>1$ by the prime number theorem, but it is not easy to prove such a result for any particular value of $c\leq 1$.
The first unconditional  result, bounding gaps between primes for some $c<1$, was proved by Erd\H os  in 1940 using the small sieve.
In 1966, Bombieri and Davenport \cite{bomdav} used the Bombieri-Vinogradov theorem   to prove this for any $c\geq \frac 12$.
In 1988 Maier \cite{maier} observed that one can easily modify this to obtain
any $c\geq \frac 12e^{-\gamma}$; and he further improved this, by combining the approaches of Erd\H os and of  Bombieri and Davenport, to obtain some bound a little smaller than $\frac 14$, in  a technical \emph{tour-de-force}.

The first big breakthrough occurred in 2005 when Goldston, Pintz and Yildirim \cite{gpy} were able to show that there are   infinitely many pairs of primes $p<q$ with $q<p+c\log p$, for \emph{any} given $c>0$. Indeed they extended their methods to show that, for any $\epsilon>0$, there are infinitely many pairs of primes $p<q$ for which
\[
q-p< (\log p)^{1/2+\epsilon}.
\]
It is their method which forms the basis of the discussion in this paper.

The earliest results on short gaps between primes were proved assuming the Generalized Riemann Hypothesis. Later unconditional results, starting with Bombieri and Davenport,  used the Bombieri-Vinogradov theorem in place of the Generalized Riemann Hypothesis. It is surprising that these tools appear in arguments about gaps between primes, since they are formulated to better understand the distribution of primes in arithmetic progressions.

Like Bombieri and Davenport, Goldston, Pintz and Yildirim showed that one can better understand small gaps between primes by obtaining strong estimates on primes in arithmetic progressions, as in the Bombieri-Vinogradov Theorem.  Even more, \textsl{assuming a strong, but widely believed, conjecture about the equi-distribution of primes in arithmetic progressions}, which extends the Bombieri-Vinogradov Theorem, one can prove Zhang's theorem with $k=5$.
Applying this result to the admissible $5$-tuple, $\{ 0,\ 2,\ 6,\ 8,\ 12\}$ implies that  there are infinitely many pairs of primes $p<q$ which differ by no more than $12$; that is, there exists a positive, even integer $2k\leq 12$ such that there are infinitely pairs of primes $p,\ p+2k$.

After Goldston, Pintz and Yildirim, most of the experts tried and failed to obtain enough of an improvement of the Bombieri-Vinogradov Theorem to deduce the existence of some finite bound $B$ such that there are infinitely many pairs of primes that differ by no more than $B$. To improve the Bombieri-Vinogradov Theorem is no mean feat and people have longed discussed ``barriers'' to obtaining such improvements.  In fact a technique to improve the Bombieri-Vinogradov Theorem had been developed by Fouvry \cite{fouvry}, and by Bombieri, Friedlander and Iwaniec \cite{bfi}, but this was neither powerful enough nor general enough to work in this circumstance.

Enter Yitang Zhang, an unlikely figure to go so much further than the experts, and to find exactly the right improvement and refinement of the Bombieri-Vinogradov Theorem to establish the existence of the elusive
bound $B$ such that there are infinitely many pairs of primes that differ by no more than $B$.  By all accounts, Zhang was a brilliant student in Beijing from 1978 to the mid-80s, finishing with a master's degree, and then working on the Jacobian conjecture for his Ph.D.~at Purdue, graduating in 1992. He did not proceed to a job in academia, working in odd jobs, such as in a sandwich shop, at a motel and as a delivery worker. Finally in 1999 he got a job at the University of New Hampshire as a lecturer. From time-to-time a lecturer devotes their energy to working on proving great results, but few have done so with such aplomb as Zhang. Not only did he prove a great result, but he did so by improving \emph{technically} on the experts, having important key ideas that they missed and developing a highly ingenious and elegant construction concerning exponential sums. Then, so as not to be rejected out of hand, he wrote his difficult paper up in such a clear manner that it could not be denied. Albert Einstein worked in a patent office, Yitang Zhang in a Subway sandwich shop; both found time, despite the unrelated calls on their time and energy, to think the deepest thoughts in science. Moreover Zhang's breakthrough came at the relatively advanced age of over 55. Truly \emph{extraordinary}.

After Zhang, a group of researchers decided to team up online to push the techniques, created by Zhang, to their limit. This was the eighth incarnation of the \emph{polymath} project, which is an experiment to see whether this sort of collaboration can help research develop beyond the traditional boundaries set by our academic culture.  The original bound of $70,000,000$ was quickly reduced, and seemingly every few weeks, different parts of Zhang's argument could be improved, so that the bound came down in to the thousands. Moreover the polymath8 researchers found variants on Zhang's argument about the distribution of primes in arithmetic progressions, that allow one to avoid some of the deeper ideas that Zhang used. These modifications   enabled your author to  give an accessible complete proof in this article.

After these  clarifications of Zhang's work, two researchers asked themselves whether the original ``set-up'' of Goldston, Pintz and Yildirim could be modified to get better results. James Maynard obtained his Ph.D. this summer at  Oxford, writing one of the finest theses in sieve theory of recent years. His thesis work equipped him perfectly to question whether the basic structure of the proof could be improved. Unbeknownst to Maynard, at much the same time (late October), one of the world's greatest living mathematicians, Terry Tao, asked himself the same question. Both found, to their surprise, that a relatively minor variant made an enormous difference, and that it was suddenly much easier to prove Zhang's Main Theorem and to go far beyond, because one can avoid having to prove any difficult new results about primes in arithmetic progressions. Moreover it is now not difficult to prove results about $m$ primes in a bounded interval, rather than just two.

\section{The distribution of primes, divisors and prime $k$-tuplets}

\subsection{The prime number theorem} As we mentioned in the previous section, Gauss observed, at the age of 16, that  ``the density of primes at around $x$ is roughly $1/\log x$'', which leads quite naturally to the conjecture that
\[
\# \{ \text{primes } p\leq x\} \approx \int_2^x \frac{dt}{\log t} \sim \frac x{\log x} \quad \text{as } x\to \infty.
\]
 (We use the symbol $A(x)\sim B(x)$ for two functions $A$ and $B$ of $x$, to mean that
 $A(x)/B(x)\to 1$ as $x\to \infty$.)  This was proved in 1896, the \emph{prime number theorem}, and the integral provides a considerably more precise approximation to the number of primes $\leq x$, than $x/\log x$. However, this integral is rather cumbersome to work with, and so it is natural to instead weight each prime with $\log p$; that is we work with
 \[
 \Theta(x):= \sum_{ \substack{p \text{ prime} \\ p\leq x}} \log p
 \]
and the prime number theorem is equivalent to
\begin{equation} \label{pnt2}
\Theta(x)\sim x\quad \text{as } x\to \infty.
\end{equation}

\subsection{The prime number theorem for arithmetic progressions, I} Any prime divisor of $(a,q)$ is an obstruction to the primality of values of the  polynomial $qx+a$, and these are the only such obstructions. The prime $k$-tuplets conjecture therefore implies that if $(a,q)=1$ then there are infinitely many primes of the form $qn+a$. This was first proved by Dirichlet in 1837. Once proved, one might ask for a more quantitative result. If we look at the primes in the arithmetic progressions $\pmod {10}$:
\begin{align*}
& 11,\ 31,\ 41,\ 61,\ 71,\ 101, \ 131, \ 151, \ 181, \ 191, \ 211, \ 241,\ldots
\\ & 3,\ 13,\ 23,\ 43,\ 53,\ 73,\ 83,\ 103,\ 113,\ 163,\ 173,\ 193,\  223,\ 233,\ldots
\\ & 7,\ 17,\ 37,\ 47,\ 67,\ 97,\ 107,\ 127,\ 137,\ 157,\ 167,\  197,\ 227,\ldots
\\ & 19,\ 29,\ 59,\ 79,\ 89,\ 109,\ 139,\  149,\  179,\ 199,\ 229,\ 239,\ldots
\end{align*}
then there seem to be roughly equal numbers in each, and this pattern persists as we look further out. Let $\phi(q)$ denote the number of $a \pmod q$ for which $(a,q)=1$ (which are the only arithmetic progressions in which there can be more than one prime), and so we expect that
 \[
 \Theta(x; q,a):= \sum_{ \substack{p \text{ prime} \\ p\leq x\\ p\equiv a \pmod q}} \log p \sim \frac x{\phi(q)} \quad \text{as } x\to \infty.
\]
This is the \emph{prime number theorem for arithmetic progressions} and was first proved by suitably modifying the proof of the prime number theorem.

The function $\phi(q)$ was studied by Euler, who showed that it is \emph{multiplicative}, that is
\[
\phi(q) = \prod_{p^e\| q} \phi(p^e)
\]
(where $p^e\|q$ means that $p^e$ is the highest power of prime $p$ dividing $q$) and, from this formula, can easily be determined for all $q$ since $\phi(p^e)=p^e-p^{e-1}$ for all $e\geq 1$.

\subsection{The prime number theorem and the M\"obius function}\label{pntMob} Multiplicative functions lie at the heart of much of the theory of the distribution of prime numbers. One, in particular, the M\"obius function, $\mu(n)$, plays a prominent role. It is defined as $\mu(p)=-1$ for every prime $p$, and $\mu(p^m)=0$ for every prime $p$ and exponent $m\geq 2$; the value at any given integer $n$ is then deduced from the values at the prime powers, by multiplicativity: If $n$ is squarefree then $\mu(n)$ equals $1$ or $-1$ according to whether $n$ has an even or odd number of prime factors, respectively.  One might guess that there are roughly equal  numbers of each, which one can phrase as the conjecture that
\[
\frac 1x \sum_{n\leq x} \mu(n) \to 0 \ \ \text{as} \ \ n\to \infty.
\]
This is a little more difficult to prove than it looks; indeed it is also equivalent to \eqref{pnt2}, the prime number theorem.  That equivalence is proved using the remarkable identity
\begin{equation} \label{VMidentity}
 \sum_{ab=n} \mu(a) \log b  \ = \ \begin{cases}
\log p &\text{ if } n=p^m, \text{ where } p \text{ is prime}, m\geq 1;\\
0  &\text{ otherwise. }
\end{cases}
\end{equation}
For more on this connection see the forthcoming book \cite{GS}.

\subsection{Recognizing prime powers and prime $k$-tuplets}  \label{Recogktuple} It is convenient to denote the right-hand side of \eqref{VMidentity} by $\Lambda(n)$ so that
\[
\Lambda(n) =  \sum_{d|n} \mu(d) \log n/d.
\]
In \eqref{VMidentity} we saw that $\Lambda(n)$  is supported  (only) on integers  $n$ that are prime powers,\footnote{By \emph{supported on} we mean ``can be non-zero only on''.} so this identity allows us to distinguish between composites and prime powers. This is useful because the functions in the summands are arithmetic functions that can be studied directly.  Such identities can be used to identify integers with no more than $k$ prime factors.  For example
\[
\Lambda_2(n) := \sum_{d|n} \mu(d) (\log n/d)^2 \ = \ \begin{cases}
(2m-1)(\log p)^2 &\text{ if } n=p^m;\\
2\log p\log q &\text{ if } n=p^aq^b,\ p\ne q;\\
0  &\text{ otherwise; }
\end{cases}
\]
that is, $\Lambda_2(n)$ is supported  (only) on integers  $n$ that have no more than two distinct prime factors.
In general (as seems to have first been discovered by Golomb \cite{golomb}),
\[
\Lambda_k(n) := \sum_{d|n} \mu(d) (\log n/d)^k
\]
is supported only when $n$ has no more than $k$ distinct prime factors (that is, $\Lambda_k(n)=0$  if  $\nu(n)>k$, where $\nu(m)$ denotes the number of distinct prime factors of $m$).  One can deduce (what at first sight seems to be a generalization) that, for any integer $R$,
\[
 \sum_{d|n} \mu(d) (\log R/d)^k
\]
is supported only when $n$ has no more than $k$ distinct prime factors.

Suppose that $0<a_1<\ldots<a_k$. We now show that if $n\geq a_k^{k-1}$ and $\Lambda_k(\mathcal P(n))\ne 0$ then
$\mathcal P(n)$ must  have \emph{exactly} $k$ distinct prime factors; moreover, if the $k$ prime factors of $\mathcal P(n)$ are  $p_1,\ldots,p_k$, then
\[
\Lambda_k(\mathcal P(n)) = k! (\log p_1)\ldots (\log p_k).
\]
\begin{proof} If $\mathcal P(n)$ has  $r\leq k-1$ distinct prime factors, call them $p_1,\ldots, p_r$. For each $p_i$ select some index $j=j(i)$ for which the power of $p_i$ dividing $n+a_j$ is maximized. Evidently there exists some $J,\ 1\leq J\leq k$ which is not a $j(i)$ for any $i$.  Therefore if $p_i^{e_i} \| n+a_J$ for each $i$ then
\[
p_i^{e_i} |  (n+a_J) - (n+a_{j(i)})  = (a_J-a_{j(i)}), \text{ which divides } \prod_{\substack{1\leq j\leq k \\ j\ne J}} (a_J-a_{j}) .
\]
Hence
\[
n+a_J = \prod_{i=1}^r \ p_i^{e_i}  \text{ divides } \ \prod_{\substack{1\leq j\leq k \\ j\ne J}} (a_J-a_{j}),
\]
and so $n<n+a_J\leq a_k^{k-1}\leq n$, by hypothesis,
which  is impossible.
\end{proof}

Selberg championed a surprising, yet key, notion of sieve theory; that the truncation
  \[  \sum_{\substack{d|n \\ d\leq R}} \mu(d) \log R/d \]
  is ``sensitive to primes'' (though not necessarily only supported on integers with few prime factors); and is considerably easier to work with in various analytic arguments.  In our case, we will work with the function
  \[ \sum_{\substack{d|\mathcal P(n) \\ d\leq R}} \mu(d) (\log R/d)^k, \]
  which is analogously ``sensitive'' to prime $k$-tuplets, and easier to work with than the full sum for $\Lambda_k(\mathcal P(n))$.

\subsection{A quantitative prime $k$-tuplets conjecture}
\label{Primektuples} We are going to develop a heuristic to guesstimate the number of pairs of twin primes $p, p+2$ up to $x$. We start with Gauss's statement that ``the density of primes at around $x$ is roughly $1/\log x$. Hence the probability that $p$ is prime is $1/\log x$, and the probability that $p+2$ is prime is $1/\log x$ so, assuming that these events are independent,  the probability that $p$ and $p+2$ are simultaneously prime is
\[ \frac 1{\log x} \cdot \frac 1{\log x} \ = \frac 1{(\log x)^2} ; \]
and so we might expect about $x/(\log x)^2$ pairs of twin primes $p,p+2\leq x$. However there is a problem with this reasoning, since we are implicitly assuming that the events ``$p$ is prime for an arbitrary integer $p\leq x$'', and ``$p+2$ is prime  for an arbitrary integer $p\leq x$'', can be considered to be independent. This is obviously false since, for example, if $p$ is even then $p+2$ must also be.\footnote{This reasoning can be seen to be false for a more dramatic reason: The analogous argument implies that there are $\sim x/(\log x)^2$ prime pairs $p,p+1\leq x$.}  So, we correct for the non-independence modulo small primes $q$, by the ratio  of the probability that both $p$ and $p+2$ are not divisible by $q$, to the probabiliity that $p$ and $p'$ are not divisible by $q$.

Now the probability that $q$ divides an arbitrary integer $p$ is $1/q$; and hence the probability that $p$ is not divisible by $q$ is $1-1/q$. Therefore the probability that both of two independently chosen integers are not divisible by $q$, is $(1-1/q)^2$.

The probability that $q$ does not divide either $p$ or $p+2$, equals the probability that $p\not\equiv 0$ or $-2 \pmod q$. If $q>2$ then $p$ can be in any one of $q-2$ residue classes mod $q$, which occurs, for a randomly chosen $p \pmod q$, with probability $1-2/q$.
 If $q=2$ then $p$ can be in any just one  residue class mod 2, which occurs with probability $1/2$. Hence the  ``correction factor'' for divisibility  by $2$ is
\[  \frac{ (1-\frac 12) } { (1-\frac 12)^2 } = 2,\]
and the ``correction factor'' for divisibility  by any prime $q>2$ is
\[  \frac{ (1-\frac 2q) } { (1-\frac 1q)^2 } .\]

Divisibility by different small primes is independent, as we vary over values of $n$, by the Chinese Remainder Theorem, and so we might expect to multiply together all of these correction factors, corresponding to each ``small'' prime $q$.  The question then becomes, what does ``small'' mean? In fact, it doesn't matter much because the product of the correction factors over larger primes is very close to 1, and hence we can simply extend the correction to be a product over all primes $q$. (More precisely, the infinite product over all $q$, converges.) Hence we define the \emph{ twin prime constant} to be
\[
C:= 2 \prod_{ \substack{q \ \text{prime} \\  q\geq 3 }}  \frac{ (1-\frac 2q) } { (1-\frac 1q)^2 }  \approx 1.3203236316,
\]
the total correction factor over all primes $q$. We then conjecture that the number of prime pairs $p,p+2\leq x$ is
\[
\sim C \frac x{(\log x)^2} .
\]
Computational evidence suggests that this is a pretty good guess. An analogous argument implies the conjecture that the number of prime pairs $p,p+2k\leq x$ is
\[
\sim C \prod_{\substack{p|k \\ p\geq 3}} \left(  \frac {p-1}{p-2} \right) \  \frac x{(\log x)^2} .
\]

This argument is easily modified to make an analogous prediction for any $k$-tuple:\ Given $a_1,\ldots, a_k$, let $\Omega(p)$ be the set of distinct residues given by $a_1,\ldots, a_k \pmod p$, and then let $\omega(p)=|\Omega(p)|$. None of the $n+a_i$ is divisible by $p$ if and only if $n$ is in any one of $p-\omega(p)$ residue classes mod $p$, and therefore the correction factor for prime $p$ is
\[  \frac{ (1-\frac {\omega(p)}p) } { (1-\frac 1p)^k } .\]
Hence we predict that  the number of prime $k$-tuplets $n+a_1,\ldots , n+a_k\leq x$ is,
\[
\sim  C(a) \frac x{(\log x)^k}  \ \text{ where } \ \  C(a):= \prod_p  \frac{ (1-\frac {\omega(p)}p) } { (1-\frac 1p)^k }.
\]
An analogous conjecture, via similar reasoning, can be made for the frequency of prime $k$-tuplets of polynomial values in several variables. What is remarkable is that computational evidence suggests that these conjectures do approach the truth, though this rests on the rather shaky theoretical framework given here. A more convincing theoretical framework based on the \emph{circle method} (so rather more difficult) was given by Hardy and Littlewood \cite{hardy}, which we will discuss in Appendix One.

 \section{Uniformity in arithmetic progressions}

\subsection{When primes are first equi-distributed in arithmetic progressions} By when are we guaranteed that the primes are more-or-less equi-distributed amongst the arithmetic progressions $a \pmod q$ with $(a,q)=1$?  That is, for what $x$ do we have
\begin{equation} \label{PNTaps}
\Theta(x;q,a) \sim \frac{x}{\phi(q)} \text{  for all   }  (a,q)=1 ?
\end{equation}
Here $x$ should be a function of $q$, and the asymptotic should hold as $q\to \infty$.

Calculations suggest  that, for any $\epsilon>0$, if $q$ is sufficiently large and $x\geq q^{1+\epsilon}$ then the primes  up to $x$ are equi-distributed amongst the arithmetic progressions $a \pmod q$ with $(a,q)=1$, that is \eqref{PNTaps} holds.  However no one has a plausible plan of how to prove such a result at the moment. The slightly weaker statement that  \eqref{PNTaps} holds for any $x\geq q^{2+\epsilon}$, can be shown to be true, assuming the Generalized Riemann Hypothesis. This gives us a clear plan for proving such a result, but one which has seen little progress in the last century!

The best unconditional results known involve much larger values of $x$, equidistribution only being proved once $x\geq e^{q^\epsilon}$. This is the \emph{Siegel-Walfisz Theorem}, and it can be stated in several (equivalent) ways  with an error term:\ For any $B>0$ we have
\begin{equation} \label{SW1}
\Theta(x;q,a) = \frac{x}{\phi(q)} +O\left( \frac x{(\log x)^B} \right) \text{   for all   }  (a,q)=1.
\end{equation}
Or: for any $A>0$ there exists $B>0$ such that if $q<(\log x)^A$ then
\begin{equation} \label{SW2}
\Theta(x;q,a) = \frac{x}{\phi(q)} \left\{ 1+O\left( \frac 1{(\log x)^B} \right)  \right\} \text{   for all   }  (a,q)=1.
\end{equation}
That $x$ needs to be so large compared to $q$ limits the applicability of this result.

The great breakthough of the second-half of the twentieth century came in appreciating that for many applications, it is not so important that we know that equidistribution holds for \emph{every} $a$ with $(a,q)=1$, and \emph{every} $q$ up to some $Q$, but rather that  it holds for \emph{most} such $q$ (with $Q=x^{1/2-\epsilon}$). It takes some juggling of variables to state the Bombieri-Vinogradov Theorem:\
We are interested, for each modulus $q$, in the size of the largest error term
\[
\max_{\substack{a\mod q \\ (a,q)=1}} \  \left| \Theta(x; q,a) - \frac x{\phi(q)} \right| ,
\]
or even
\[
\max_{y\leq x} \max_{\substack{a\mod q \\ (a,q)=1}} \  \left| \Theta(y; q,a) - \frac y{\phi(q)} \right| .
\]
The  bounds
$- \frac x{\phi(q)} \leq \Theta(x; q,a)- \frac x{\phi(q)}\leq (\frac xq+1) \log x$ are trivial, the upper bound obtained by bounding the possible contribution from each term of the arithmetic progression.
%
 We would like to improve on these bounds, perhaps by a power of $\log x$ (as in \eqref{SW1}), but we are unable to do so for all $q$. However, what we can prove is that   \emph{exceptional $q$  are few and far between},\footnote{\textsl{Exceptional} $q$ being those $q$ for which
 $|\Theta(x; q,a)- \frac x{\phi(q)}|$ is not small, for some $a$ coprime to $q$.} and the Bombieri-Vinogradov Theorem expresses this in a useful form.  The  ``trivial'' upper bound, obtained by adding up the above quantities over all $q\leq Q<x$, is
\[
\sum_{q\leq Q} \  \max_{\substack{a\mod q \\ (a,q)=1}} \  \left| \Theta(x; q,a) - \frac x{\phi(q)} \right|  \leq  \sum_{q\leq Q} \left( \frac {2x}{q} \log x + \frac x{\phi(q)} \right)\ll   x(\log x)^2.
\]
(Throughout, the symbol ``$\ll$'', as in ``$f(x) \ll  g(x)$'' means ``there exists a constant $c>0$ such that  $f(x) \leq cg(x)$.'')
The Bombieri-Vinogradov states that we can beat this trivial bound by an arbitrary power of $\log x$, provided $Q$ is a little smaller than $\sqrt{x}$:

\textbf{The Bombieri-Vinogradov Theorem}. \emph{For any given $A>0$ there exists a constant $B=B(A)$, such that
\[
\sum_{q\leq Q} \  \max_{\substack{a\mod q \\ (a,q)=1}} \  \left| \Theta(x; q,a) - \frac x{\phi(q)} \right| \ll_A \frac x{ (\log x)^A}
\]
where $Q=x^{1/2}/(\log x)^B$.}

In fact one can take $B=2A+5$; and one can also replace the summand here by the expression above with the maximum over $y$ (though we will not need to use this here).

\subsection{Breaking the $x^{1/2}$-barrier}
It is believed that estimates like that in the  Bombieri-Vinogradov Theorem hold with $Q$ significantly larger than $\sqrt{x}$; indeed Elliott and Halberstam conjectured  \cite{elliott} that one can take $Q=x^c$ for any constant $c<1$:

\textbf{The Elliott-Halberstam conjecture} \emph{For any given $A>0$ and $\eta,\ 0<\eta<\frac 12$, we have
\[
\sum_{q\leq Q} \  \max_{\substack{a\mod q \\ (a,q)=1}} \  \left| \Theta(x; q,a) - \frac x{\phi(q)} \right| \ll  \frac x{ (\log x)^A}
\]
where $Q=x^{1/2+\eta}$.}

However, it was shown in \cite{fg-1} that one \emph{cannot} go so far as to take $Q=x/(\log x)^B$.

This conjecture was the starting point for the work of Goldston, Pintz and Y{\i}ld{\i}r{\i}m \cite{gpy}, that was used by Zhang \cite{zhang} (which we give in detail in the next section). It can be applied to obtain the following result, which we will prove.

\begin{theorem}[Goldston-Pintz-Y{\i}ld{\i}r{\i}m]\label{gpy-thm}{\ \cite{gpy}}  Let $k \geq 2$, $l \geq 1$ be  integers, and  $0 < \eta < 1/2$, such that
\begin{equation}\label{thetal}
 1+2\eta > \left(1 + \frac{1}{2l+1}\right) \left(1 + \frac{2l+1}{k}\right).
\end{equation}
Assume that the Elliott-Halberstam conjecture holds with $Q=x^{1/2+\eta}$. If $x+a_1,\ldots, x+a_k$ is an admissible  set of forms  then there are infinitely many integers $n$ such that at least two of  $n+a_1,\ldots, n+a_k$ are prime numbers.
\end{theorem}

 The conclusion here is exactly the statement of Zhang's main theorem.

 If the Elliott-Halberstam conjecture conjecture holds for some $\eta>0$ then select $l$  to be an integer so large that  $ \left(1 + \frac{1}{2l+1}\right) < \sqrt{1+2\eta}$.  Theorem \ref{gpy-thm} then implies Zhang's theorem for $k=(2l+1)^2$.

The Elliott-Halberstam conjecture seems to be too difficult to prove for now, but progress has been made when restricting to one particular residue class:\ Fix integer $a\ne 0$. We believe that for any fixed  $\eta,\ 0<\eta<\frac 12$, one has
\[
\sum_{\substack{q\leq Q \\ (q,a)=1}} \   \  \left| \Theta(x; q,a) - \frac x{\phi(q)} \right| \ll  \frac x{ (\log x)^A}
\]
where $Q=x^{1/2+\eta}$, which follows from the Elliott-Halberstam conjecture (but is weaker).

The key to progress has been to notice that if one can``factor'' the key terms here  then the extra flexibility allows one to make headway.
For example by factoring the modulus $q$ as, say, $dr$ where $d$ and $r$ are roughly some pre-specified sizes.  The simplest class of integers $q$ for which this can be done is the \emph{$y$-smooth integers}, those integers whose prime factors are all $\leq y$.  For example if we are given  a $y$-smooth integer $q$ and we want $q=dr$ with $d$ not much smaller than $D$, then we select $d$ to be the largest divisor of $q$ that is $\leq D$ and we see that $D/y<d\leq D$. This is precisely the class of moduli that Zhang considered.

The other ``factorization'' concerns the sum $\Theta(x;q,a)$. The terms of this sum can be written as a sum of products, as we saw in \eqref{VMidentity}; in fact we will decompose  this further,  partitioning the values of $a$ and $b$ (of \eqref{VMidentity})  into different ranges.

\begin{theorem}[Yitang Zhang's Theorem]\label{Zhangthm} There exist constants $\eta,\delta>0$ such that for any given integer $a$, we have
\begin{equation} \label{EHsmooth}
\sum_{\substack{q\leq Q \\ (q,a)=1 \\ q \text{ is } y-\text{smooth} \\ q \text{ squarefree} }} \   \  \left| \Theta(x; q,a) - \frac x{\phi(q)} \right| \ll_A \frac x{ (\log x)^A}
\end{equation}
where $Q=x^{1/2+\eta}$ and $y=x^\delta$.
\end{theorem}

Zhang \cite{zhang} proved his Theorem for $\eta/2=\delta=\frac{1}{1168}$, and his argument works provided $414 \eta+ 172 \delta < 1$. We will prove this result, by a somewhat simpler proof, provided $162 \eta+ 90 \delta < 1$, and the more sophisticated proof of \cite{polymath8} gives \eqref{EHsmooth} provided $43\eta+27\delta<1$.
 We expect  that this estimate holds for every $\eta\in [0,1/2)$ and every $\delta\in (0,1]$, but just proving it for any positive pair $\eta,\delta>0$ is an extraordinary breakthrough that has an enormous effect on number theory, since it is such an applicable result (and technique).  This is the technical result that truly lies at the heart of Zhang's result about bounded gaps between primes, and sketching a proof of this is the focus of the second half of this article.

   \section{Goldston-Pintz-Y{\i}ld{\i}r{\i}m's argument}\label{gpy-sec}

The combinatorial argument of Goldston-Pintz-Y{\i}ld{\i}r{\i}m \cite{gpy} lies at the heart of the proof that there are bounded gaps between primes. (Henceforth we will call it ``the GPY argument''. See \cite{sound} for a more complete discussion of their ideas.)

\subsection{The set up}
  Let ${\mathcal H} = (a_1< a_2< \ldots < a_k)$ be an admissible $k$-tuple, and take $x>a_k$.    Our goal is to select a weight   for which $\text{weight}(n)\geq 0$ for all $n$, such that
\begin{equation} \label{gpy1}
\sum_{x<n\leq 2x} \text{weight}(n) \left(\sum_{i=1}^{k} \theta(n+a_i) - \log 3x\right) > 0,
\end{equation}
where $\theta(m)=\log m$ if $m=p$ is prime, and $\theta(m)=0$ otherwise.
If we can do this then  there must exist an integer $n$ such that
$$ \text{weight}(n) \left(\sum_{i=1}^{k} \theta(n+a_i) - \log 3x\right) >0.$$
In that case $\text{weight}(n)\ne 0$ so that  $\text{weight}(n)>0$, and therefore
\[
\sum_{i=1}^{k} \theta(n+a_i) > \log 3x.
\]
However each $n+a_i\leq 2x+a_k<2x+x$ and so each $\theta(n+a_i)<\log 3x$.  This implies that at least two of the $\theta(n+a_i)$ are non-zero, that is, at least two of $n+a_1,\ldots,n+a_k$ are prime.

A simple idea, but the difficulty comes in selecting the function $\text{weight}(n)$ with these properties in such a way that we can evaluate the sums in \eqref{gpy1}.  Moreover in \cite{gpy} they also require that $\text{weight}(n)$ is sensitive to when each $n+a_i$ is  ``almost prime''. All of these properties can be acquired by using a construction championed by Selberg. In order that $\text{weight}(n)\geq 0$ one can simply take it to be a square. Hence we select
\[
\text{weight}(n) := \left( \sum_{\substack{d|\mathcal P(n) \\ d\leq R}} \lambda(d) \right)^2,
\]
where the sum is over the positive integers $d$ that divide $\mathcal P(n)$, and
\[
\lambda(d):=  \mu(d)  G \left(\frac{\log d}{\log R}\right) ,
\]
where $G(.)$ is a measurable, bounded function, supported only on $[0,1]$, and $\mu$ is the M\"obius function. Therefore $\lambda(d)$ is   supported only on  squarefree, positive integers, that are $\leq R$.
(This generalizes the discussion at the end of section \ref{Recogktuple}.)

We can select  $G(t)=(1-t)^m/m!$ to obtain the results of  this section but it will pay, for  our understanding of the Maynard-Tao construction, if we prove the GPY result for  more general $G(.)$.

\subsection{Evaluating the sums over $n$} Expanding the above sum gives
\begin{equation} \label{gpy2}
\sum_{\substack{d_1,d_2\leq R \\ D:=[d_1,d_2]}} \lambda(d_1)\lambda(d_2)
  \left(\sum_{i=1}^{k} \sum_{\substack{x<n\leq 2x\\ D|\mathcal P(n)}}  \theta(n+a_i) - \log 3x \sum_{\substack{x<n\leq 2x\\ D|\mathcal P(n)}} 1 \right) .
\end{equation}
Let $\Omega(D)$ be the set of congruence classes $m \pmod D$ for which $D|P(m)$; and let
$\Omega_i(D)$ be the set of congruence classes $m\in \Omega(D)$ with $(D,m+a_i)=1$. Hence the parentheses in the above line equals
\begin{equation} \label{gpy3}
\sum_{i=1}^{k} \sum_{m\in \Omega_i(D)}   \sum_{\substack{x<n\leq 2x\\ n\equiv m \pmod D}}  \theta(n+a_i) - \log 3x \sum_{m\in \Omega(D)}   \sum_{\substack{x<n\leq 2x \\ n\equiv m \pmod D}}  1,
\end{equation}
since $P(n)\equiv P(m) \pmod D$ whenever $n\equiv m \pmod D$.

Our first goal  is to evaluate the sums over $n$. The final sum is easy; there are
$x/D+O(1)$ integers in a given arithmetic progression with difference $D$, in an interval of length $x$.
Here $D:=[d_1,d_2]\leq d_1d_2\leq R^2$, and so the error term here is much smaller than the main term if $R^2$ is much smaller than $x$. We will select $R\leq x^{\frac 12 -o(1)}$ so that the sum of all of these error terms will be irrelevant to the subsequent calculations.

Counting the number of primes in a given arithmetic progression with difference $D$, in an interval of length $x$. is much more difficult.   We expect that  \eqref{PNTaps} holds, so that
each
\[
\Theta(2x;D,m+a_i)-\Theta(x;D,m+a_i)  \sim \frac x{\phi(D)} .
\]
The error terms here are larger and more care is needed. The sum of all of these error terms will be small enough to ignore, provided that the error terms are smaller than the main terms by an arbitrarily large power of $\log x$, at least on average.  This shows why the Bombieri-Vinogradov Theorem is so useful, since it implies the needed estimate provided $D<x^{1/2-o(1)}$ (which follows if $R<x^{1/4-o(1)}$).  Going any further is difficult, so that the $\frac 14$ is an important barrier. Goldston, Pintz and Y{\i}ld{\i}r{\i}m showed that if one can go just beyond $\frac 14$ then one can prove that there are bounded gaps between primes, but there did not seem to be any techniques available to them to do so.

For the next part of this discussion we'll ignore these accumulated error terms, and estimate the size of the sum of the main terms. First, though, we need to better understand the sets $\Omega(D)$  and  $\Omega_i(D)$. These sets may be constructed using the Chinese Remainder Theorem from the sets with $D$ prime. Therefore if $\omega(D):=|\Omega(D)|$ then $\omega(.)$ is a multiplicative function. Moreover each $|\Omega_i(p)|=\omega(p)-1$, which we denote by $\omega^*(p)$, and each $|\Omega_i(D)|=\omega^*(D)$, extending $\omega^*$ to be  a multiplicative function. Putting this altogether we obtain in \eqref{gpy3} a main term of
\[
  k \omega^*(D)  \frac x{\phi(D)} - (\log 3x)   \omega(D) \frac xD =
  x\left( k   \frac {\omega^*(D)}{\phi(D)} - (\log 3x)   \frac { \omega(D)}D \right) .
 \]
This is typically negative which explains why we cannot simply take the $\lambda(d)$ to all be positive in \eqref{gpy2}.  Substituting this main term for \eqref{gpy3} into each summand of \eqref{gpy2} we  obtain,
\begin{equation} \label{gpy4}
  x\left( k  \sum_{\substack{d_1,d_2 \leq R \\ D:=[d_1,d_2]}} \lambda(d_1)\lambda(d_2)  \frac {\omega^*(D)}{\phi(D)} - (\log 3x)  \sum_{\substack{d_1,d_2\leq R \\ D:=[d_1,d_2]}} \lambda(d_1)\lambda(d_2)  \frac { \omega(D)}D \right) .
\end{equation}

The two sums over $d_1$ and $d_2$ in \eqref{gpy4} are not easy to evaluate: The use of the M\"obius function leads to many terms being positive, and many negative, so that   there is  a lot of cancelation.  There are several techniques in analytic number theory that allow one to get accurate estimates for such sums,   two more analytic (\cite{gpy}, \cite{polymath8b}), the other more combinatorial  (\cite{sound}, \cite{ggpy}). We will discuss them all.

\subsection{Evaluating the sums using Perron's formula} Perron's formula allows one to study inequalities using complex analysis:
\[
\frac 1{2i\pi}  \int_{\text{Re}(s)=2} \frac{y^s}{s} \ ds =
\begin{cases}  1 &\text{if} \ y>1; \\
 1/2 &\text{if} \ y=1 ;\\
0 &\text{if} \ 0<y<1.
\end{cases}
\]
(Here the subscript ``$\text{Re}(s)=2$'' means that we integrate along the line $s:\ \text{Re}(s)=2$; that is  $s=2+it$, as $t$ runs from $-\infty$ to $+\infty$.)
So to determine whether $d<R$ we simply compute this integral with $y=R/d$. (The  special case, $d=R$, has a negligible effect on our sums, and can be avoided by selecting $R\not\in \mathbb Z$). Hence  the second sum in \eqref{gpy4} equals
\[
\sum_{\substack{d_1,d_2\geq 1 \\ D:=[d_1,d_2]}} \lambda(d_1)\lambda(d_2)  \frac { \omega(D)}D  \cdot
\frac 1{2i\pi}  \int_{\text{Re}(s_1)=2} \frac{(R/d_1)^{s_1}}{s_1} \ d{s_1} \cdot
\frac 1{2i\pi}  \int_{\text{Re}(s_2)=2} \frac{(R/d_2)^{s_2}}{s_2} \ d{s_2} .
\]
Re-organizing this we obtain
\begin{equation} \label{1stIntegral}
\frac 1{(2i\pi)^2}  \int_{\substack{\text{Re}(s_1)=2 \\ \text{Re}(s_2)=2}}
\ \left( \sum_{\substack{d_1,d_2\geq 1 \\ D:=[d_1,d_2]}} \frac{\lambda(d_1)\lambda(d_2)}{d_1^{s_1}d_2^{s_2} } \frac { \omega(D)}D \right) \
R^{s_1+s_2} \frac{d{s_2}}{s_2} \cdot \frac{d{s_1}}{s_1}
\end{equation}
We will compute the sum in the middle in the special case that $\lambda(d)=\mu(d)$, the more general case following from a variant of this argument. Hence we have
\begin{equation} \label{1stSum}
\sum_{ d_1,d_2\geq 1  } \frac{\mu(d_1)\mu(d_2)}{d_1^{s_1}d_2^{s_2} } \frac { \omega([d_1,d_2])}{[d_1,d_2]} .
\end{equation}
The summand is a multiplicative function, which means that we can evaluate it prime-by-prime. For any given prime $p$,  the summand is $0$ if $p^2$ divides $d_1$ or $d_2$ (since then $\mu(d_1)=0$ or $\mu(d_2)=0$). Therefore we have only four cases to consider: $p\nmid d_1,d_2; \ p|d_1, p\nmid  d_2; \ p\nmid  d_1, p|d_2; \ p|d_1, p|d_2$, so the $p$th factor is
\[
1 - \frac{1}{p^{s_1}} \cdot \frac{\omega(p)}p- \frac{1}{p^{s_2}} \cdot \frac{\omega(p)}p+ \frac{1}{p^{s_1+s_2}} \cdot \frac{\omega(p)}p .
\]
We have seen that $\omega(p)=k$ for all sufficiently large $p$  so, in that case, the above becomes
\begin{equation}\label{withk}
1 -  \frac{k}{p^{1+s_1}}   - \frac{k}{p^{1+s_2}} + \frac{k}{p^{1+s_1+s_2}}   .
\end{equation}
In the analytic approach, we compare the integrand to a (carefully selected) power of the Riemann-zeta function, which is defined as
\[
\zeta(s)=\sum_{n\geq 1} \frac 1 {n^s} \ = \ \prod_{p \ \text{prime}} \left( 1 - \frac 1 {p^s} \right)^{-1} \  \text{for Re}(s)>1
.
\]
The $p$th factor of $\zeta(s)$ is $\left( 1 -\frac{1}{p^{s}}\right)^{-1}$ so, as a first approximation, \eqref{withk} is roughly
\[
\left( 1 -\frac{1}{p^{1+s_1+s_2}}\right)^{-k} \left( 1 -\frac{1}{p^{1+s_1}}\right)^{k} \left( 1 -\frac{1}{p^{1+s_2}}\right)^{k} .
\]
Substituting this back into \eqref{1stIntegral} we obtain
\[
\frac 1{(2i\pi)^2}  \int\int_{\substack{\text{Re}(s_1)=2 \\ \text{Re}(s_2)=2}} \ \
\frac{  \zeta(1+s_1+s_2)^k } {  \zeta(1+s_1)^k  \zeta(1+s_2)^k}  G(s_1,s_2) \ \
R^{s_1+s_2} \frac{d{s_2}}{s_2} \cdot \frac{d{s_1}}{s_1} .
\]
where
\[
G(s_1,s_2):= \prod_{p \ \text{prime}}
\left( 1 -\frac{1}{p^{1+s_1+s_2}}\right)^{k} \left( 1 -\frac{1}{p^{1+s_1}}\right)^{-k} \left( 1 -\frac{1}{p^{1+s_2}}\right)^{-k} \left( 1 -  \frac{\omega(p)}{p^{1+s_1}}   - \frac{\omega(p)}{p^{1+s_2}} + \frac{\omega(p)}{p^{1+s_1+s_2}} \right) .
\]
To determine the value of this integral we move both contours in the integral slightly to the left of the lines Re$(s_1)=$Re$(s_2)=0$, and show that the main contribution comes, via Cauchy's Theorem, from the pole at $s_1=s_2=0$. This can be achieved using our understanding of the Riemann-zeta function, and by noting that
\[
G(0,0):= \prod_{p \ \text{prime}} \left( 1 -  \frac{\omega(p)}{p}  \right)
 \left( 1 -\frac{1}{p}\right)^{-k}  = C(a) \ne 0.
\]
Remarkably when one does the analogous calculation with the first sum in \eqref{gpy4}, one takes $k-1$ in place of $k$, and then
\[
G^*(0,0):= \prod_{p \ \text{prime}} \left( 1 -  \frac{\omega^*(p)}{p-1}  \right)
 \left( 1 -\frac{1}{p}\right)^{-(k-1)}  = C(a) ,
\]
also.  Since it is so unlikely that these two quite different products give the same constant by co-incidence, one can feel sure that the method is correct!

This was the technique used in  \cite{gpy} and, although the outline of the method is quite compelling, the details of the contour shifting can be complicated.

\subsection{Evaluating the sums using Fourier analysis}\label{sect:TaoApproach}  Both analytic approaches depend on the simple pole of the Riemann zeta function at $s=1$. The Fourier analytic approach (first used, to my knowledge, by Green and Tao, and in this context, on Tao's blog) avoids some of the more mysterious, geometric technicalities (which emerge when shifting contours in high dimensional space), since the focus is more on the pole itself.

To appreciate the method we   prove a fairly general result, starting with smooth functions
$F, H: [0,+\infty)\to \R$ that are supported on the finite interval $[0,\frac{\log R}{\log x}]$, and then letting
$\lambda(d_1)=\mu(d_1)F(\frac{\log d_1}{\log x})$ and $\lambda(d_2)=\mu(d_2)H(\frac{\log d_2}{\log x})$.
Note that $\lambda(d_1)$ is supported only when $d_1\leq R$, and similarly $\lambda(d_2)$.
Our goal is to evaluate the sum
\begin{equation}\label{EvalTao1}
 \sum_{\substack{d_1,d_2\geq 1 \\ D:=[d_1,d_2]}} \mu(d_1)\mu(d_2) \ F(\frac{\log d_1}{\log x}) H(\frac{\log d_2}{\log x})  \frac { \omega(D)}D .
\end{equation}

The function $e^t F(t)$ (and similarly $e^t H(t)$) is also a smooth, finitely supported, function.
Such a function  has  a Fourier expansion
\[
 e^v F(v) = \int_{t=-\infty}^\infty e^{-itv} f(t)\ dt \ \  \text{ and so}\ \ \ F(\frac{\log d_1}{\log x})= \int_{t=-\infty}^\infty \frac{f(t)}{d_1^{\frac{1+it}{\log x}}}\ dt,
\]
for some smooth, bounded function $f:\R \to \C$ that is rapidly decreasing,\footnote{That is, for any given $A>0$ there exists a constant $c_A$ such that $|f(\xi)|\leq c_A/|\xi|^{-A} $.} and therefore the tail of the integral (for instance when $|t|>\sqrt{\log x}$) does not contribute much. Substituting this and the analogous formula for $H(.)$ into \eqref{EvalTao1}, we obtain
\[
\int_{t,u=-\infty}^\infty f(t) h(u) \sum_{\substack{d_1,d_2\geq 1 \\ D:=[d_1,d_2]}} \frac{\mu(d_1)\mu(d_2)}{ d_1^{\frac{1+it}{\log x}} d_2^{\frac{1+iu}{\log x}}}  \frac { \omega(D)}D  \ dt du .
\]
We evaluated this same sum,  \eqref{1stSum} with $s_1=\frac{1+it}{\log x}$ and $s_2=\frac{1+iu}{\log x}$, in the previous approach, and so know that our integral equals
\[
\int_{t,u=-\infty}^\infty f(t) h(u)\frac{  \zeta(1+\frac{2+i(t+u)}{\log x})^k } {  \zeta(1+\frac{1+it}{\log x})^k  \zeta(1+\frac{1+iu}{\log x})^k}  G(\frac{1+it}{\log x},\frac{1+iu}{\log x})   \ dt du .
\]
One can show that the contribution with $|t|, |u|>\sqrt{\log x}$ does not contribute much since $f$ and $h$ decay so rapidly. When
$|t|, |u|\leq \sqrt{\log x}$ we are near the  pole of $\zeta(s)$, and can get very good approximations for the zeta-values by using the Laurent expansion $\zeta(s)=1/s+O(1)$. Moreover $G(\frac{1+it}{\log x},\frac{1+iu}{\log x}) =G(0,0)+o(1)$ using its Taylor expansion, and therefore our integral is very close to
\[
\frac{G(0,0)}{(\log x)^k} \ \int_{t,u=-\infty}^\infty f(t) h(u)\frac{(1+it)(1+iu)} {2+i(t+u)}   \ dt du .
\]

 By the definition of $F$ we have $F'(v) = -\int_{t=-\infty}^\infty (1+it)e^{-(1+it)v} f(t)\ dt$, and so
 \begin{align*}
 \int_{v=0}^\infty F'(v)H'(v)dv &= \int_{t,u=-\infty}^\infty f(t) h(u) (1+it)(1+iu)\left(  \int_{v=0}^\infty e^{-(1+it)v-(1+iu)v}  dv \right)dt du \\
 &= \int_{t,u=-\infty}^\infty f(t) h(u) \frac{(1+it)(1+iu)} {2+i(t+u)}dt du .
\end{align*}
Combining the last two displayed equations, and remembering that $G(0,0)=C(a)$ we deduce that
\eqref{EvalTao1} is asymptotically equal to
\[
\frac{C(a)}{(\log x)^k} \ \int_{v=0}^\infty F'(v)H'(v)dv.
\]
One can do the analogous calculation with the first sum in \eqref{gpy4}, taking $k-1$ in place of $k$, and obtaining the constant $G^*(0,0)=C(a)$.

\subsection{Evaluating the sums using Selberg's combinatorial approach, I} As discussed, the difficulty in evaluating the sums in \eqref{gpy4} is that there are many positive terms and many negative terms. In developing his upper bound sieve method, Selberg encountered a similar problem and dealt with it in a surprising way, using combinatorial identities to remove this issue. The method rests on a \emph{reciprocity law}:\ Suppose that $L(d)$ and $Y(r)$ are sequences of numbers, supported only on the squarefree integers. If
\[
Y(r):=\mu(r)    \sump_{m:\ r|m} L(m)   \ \text{for all} \ r\geq 1,
\]
then
\[
L(d) = \mu(d)    \sump_{n:\ d|n} Y(n)  \ \text{for all} \ d\geq 1
\]
From here on, $\sump$ denotes the restriction to squarefree integers that are $\leq R$.
\footnote{Selberg developed similar ideas in his construction of a small sieve, though he neither formulated a reciprocity law, nor applied his ideas to the question of small gaps between primes.}

Let $\phi_\omega$ be the multiplicative function (defined here, only on squarefree integers) for which $\phi_\omega(p)=p-\omega(p)$. We apply the above reciprocity law with
\[
L(d):=  \frac{\lambda(d)\omega(d)}{d}    \ \ \text{ and} \ \ \
Y(r):= \frac{y(r)\omega(r)}{\phi_\omega(r)} .
\]
Now since $d_1d_2=D(d_1,d_2)$ we have
\[
\lambda(d_1)\lambda(d_2)  \frac { \omega(D)}D = L(d_1)L(d_2) \ \frac{(d_2,d_2)}{\omega((d_2,d_2))}
\]
and therefore
\[
S_1:= \sump_{\substack{d_1,d_2 \\ D:=[d_1,d_2]}} \lambda(d_1)\lambda(d_2)  \frac { \omega(D)}D  = \sum_{r, s}  Y(r)Y(s)
\sump_{\substack{d_1,d_2 \\ d_1|r,\ d_2|s}}   \mu(d_1)   \mu(d_2)   \frac {(d_1,d_2)}{ \omega((d_1,d_2))} .
\]
The summand (of the inner sum)   is multiplicative and so we can work out its value, prime-by-prime. We see that if $p|r$ but $p\nmid s$ (or vice-versa) then the sum is $1-1=0$. Hence if the sum is non-zero then $r=s$ (as $r$ and $s$ are both squarefree). In that case, if $p|r$ then the sum is $1-1-1+p/\omega(p) = \phi_\omega(p)/\omega(p)$. Hence the sum becomes
\begin{equation}\label{squaresum}
S_1=\sum_{r }  Y(r)^2 \frac{\phi_\omega(r)}{\omega(r)} = \sum_{r } \frac{y(r)^2\omega(r)}{\phi_\omega(r)}.
\end{equation}
We will select
\[
y(r):= F\left( \frac{\log r}{\log R} \right)
\]
when $r$ is squarefree, where $F(t)$ is measurable and supported only on $[0,1]$; and $y(r)=0$ otherwise. Hence we now have a sum with all positive terms so we do not have to fret about complicated cancelations.

\subsection{Sums of multiplicative functions} An important theme in analytic number theory is to understand the behaviour of sums of multiplicative functions, some being easier than others. Multiplicative functions $f$ for which the $f(p)$ are fixed, or almost fixed, were the first class of non-trivial sums to be determined. Indeed from the  Selberg-Delange theorem,\footnote{This also follows from the relatively easy proof of Theorem 1.1 of  \cite{IK}.} one can deduce that
\begin{equation} \label{SD}
\sum_{n\leq x} \frac{g(n)}n \sim \kappa(g) \cdot \frac{ (\log x)^k } {k!},
\end{equation}
where
\[
\kappa(g):=  \prod_{p \ \text{prime}} \left( 1 +\frac{g(p)}{p} +  \frac{g(p^2)}{p^{2}} +\ldots \right)
 \left( 1 -\frac{1}{p}   \right)^k
\]
when $g(p)$ is typically ``sufficiently close'' to some given positive integer $k$ that the Euler product converges. Moreover, by partial summation, one deduces that
\begin{equation} \label{SD+}
\sum_{n\leq x} \frac{g(n)}n  F\left( \frac{\log n}{\log x} \right)  \sim \kappa(g) (\log x)^k \cdot \int_0^1
F(t) \frac{t^{k-1}}{(k-1)!} dt.
 \end{equation}
We apply this in the sum above, noting that here $\kappa(g)=1/C(a)$, to obtain
\[
C(a)S_1=C(a) \ \sum_{r } \frac{\omega(r)}{\phi_\omega(r)}F\left( \frac{\log r}{\log R} \right)^2 \sim
 (\log R)^k \cdot \int_0^1 F(t)^2 \frac{t^{k-1}}{(k-1)!} dt.
\]
A similar calculation reveals that
\[
C(a)\lambda(d) \sim \mu(d)  \cdot (1-v_d)^k \int_{v_d}^1 F(t) \frac{t^{k-1}}{(k-1)!} dt \cdot (\log R)^k,
\]
where $v_d:=\frac{\log d}{\log R}$.

\begin{remark} If, as in section \ref{sect:TaoApproach}, we replace $\lambda(d_1)\lambda(d_2)$ by
$\lambda_1(d_1)\lambda_2(d_2)$ then we obtain
\[
C(a)S_1\sim (\log R)^k \cdot \int_0^1 F_1(t)F_2(t) \frac{t^{k-1}}{(k-1)!} dt .
\]
\end{remark}

\subsection{Selberg's combinatorial approach, II}
 A completely analogous calculation, but now applying the reciprocity law with
 \[
L(d):=  \frac{\lambda(d)\omega^*(d)}{\phi(d)}    \ \ \text{ and} \ \ \
Y(r):= \frac{y^*(r)\omega^*(r)}{\phi_\omega(r)} ,
\]
yields that
\begin{equation} \label{solve2}
S_2:=  \sump_{\substack{d_1,d_2 \\ D:=[d_1,d_2]}} \lambda(d_1)\lambda(d_2)  \frac {\omega^*(D)}{\phi(D)} =
  \sum_{r}  \frac{y^*(r)^2\omega^*(r)}{\phi_\omega(r)} .
\end{equation}
We need to determine $y^*(r)$ in terms of the $y(r)$, which we achieve by applying the reciprocity law twice:
\begin{align*}
y^*(r) &=\mu(r) \frac{\phi_\omega(r)}{ \omega^*(r)}  \sum_{d:\ r|d} \frac{\omega^*(d)}{\phi(d)}
 \mu(d) \frac{d}{\omega(d)} \sum_{n:\ d|n} \frac{y(n)\omega(n)}{\phi_\omega(n)} \\
 &= \frac r{\phi(r)}   \sum_{n:\ r|n} \frac{y(n)}{\phi_\omega(n/r)}
 \sum_{d:\  d/r|n/r} \mu(d/r) \frac{\omega^*(d/r)d/r}{\phi(d/r)}
    \omega(n/d) \\
    & = r \sump_{n:\ r|n}  \frac{ y(n) } {\phi(n) } = \frac r{\phi(r)} \sump_{m:\ (m,r)=1}  \frac{ y(mr) } {\phi(m) } \\
        & \sim \int_{\frac{\log r}{\log R}}^1 F(t) dt \cdot \log R,
\end{align*}
where the last estimate was obtained by applying \eqref{SD+}   with $k=1$, and taking care with the Euler product.

We now can insert this into \eqref{solve2}, and  apply  \eqref{SD+}   with $k$ replaced by $k-1$, noting that $\kappa(g^*)=1/C(a)$, to obtain
\[
C(a)S_2=C(a) \sum_{r}  \frac{y^*(r)^2\omega^*(r)}{\phi_\omega(r)} \sim (\log R)^{k+1} \cdot \int_0^1
\left(\int_t^1 F(u)du \right)^2\frac{t^{k-2}}{(k-2)!} dt.
 \]

\begin{remark} If, as in section \ref{sect:TaoApproach}, we replace $\lambda(d_1)\lambda(d_2)$ by
$\lambda_1(d_1)\lambda_2(d_2)$ then we obtain
\[
C(a)S_2 \sim (\log R)^{k+1} \cdot \int_0^1
\left(\int_t^1 F_1(u)du \right)\left(\int_t^1 F_2(v)dv \right)\frac{t^{k-2}}{(k-2)!} dt.
\]
\end{remark}

\subsection{Finding a positive difference; the proof of Theorem \ref{gpy-thm}}  From these estimate, we deduce that   $C(a)$ times \eqref{gpy4} is asymptotic to $x(\log 3x)  (\log R)^k$ times
\begin{equation} \label{CollectedUp}
k\ \frac{\log R}{\log 3x} \cdot \int_0^1
\left(\int_t^1 F(u)du \right)^2\frac{t^{k-2}}{(k-2)!} dt -
  \int_0^1 F(t)^2 \frac{t^{k-1}}{(k-1)!} dt.
\end{equation}
Assume that the Elliott-Halberstam conjecture holds with exponent $\frac 12+\eta$, and let
$R=\sqrt{Q}$. This then equals
\[
 \int_0^1 F(t)^2 \frac{t^{k-1}}{(k-1)!} dt  \cdot \left(  \frac 12 \left( \frac 12+\eta \right) \ \rho_k(F) - 1 \right)
\]
where
\begin{equation} \label{Rhok}
\rho_k(F):=k  \int_0^1 \left(\int_t^1 F(u)du \right)^2\frac{t^{k-2}}{(k-2)!} dt \bigg/
  \int_0^1 F(t)^2 \frac{t^{k-1}}{(k-1)!} dt .
\end{equation}
Therefore if
\[
\frac 12 \left( \frac 12+\eta \right) \ \rho_k(F) > 1
\]
for some $F$ that satisfies the above hypotheses, then \eqref{CollectedUp} is $>0$, which implies that  \eqref{gpy4}  is $>0$, and so \eqref{gpy1} is also $>0$, as desired.

We now need to select a suitable function $F(t)$ to proceed. A good choice is $F(t)=\frac{(1-t)^\ell}{\ell!}$.
Using the beta integral identity
\[ \int_0^1   \frac{v^{k}}{k!} \frac{(1-v)^\ell}{\ell !} dv=\frac{1 }{(k+\ell+1)!} ,\]
we obtain
\[
 \int_0^1 F(t)^2 \frac{t^{k-1}}{(k-1)!} dt  = \int_0^1 \frac{(1-t)^{2\ell}}{\ell!^2} \frac{t^{k-1}}{(k-1)!} dt
 = \frac{1  }{(k+2\ell)!} \binom{2\ell}{\ell} ,
\]
and
\[
\int_0^1 \left(\int_t^1 F(u)du \right)^2\frac{t^{k-2}}{(k-2)!} dt
= \int_0^1 \left( \frac{(1-t)^{\ell+1}}{\ell+1}\right)^2 \frac{t^{k-2}}{(k-2)!} dt
 = \frac{ 1 }{(k+2\ell+1)! } \binom{2\ell+2}{\ell+1}.
 \]
Therefore  \eqref{Rhok} is $>0$ if \eqref{thetal} holds, and so we deduce Theorem \ref{gpy-thm}.
\bigskip

To summarize: \ If the Elliott-Halberstam conjecture holds with exponent $\frac 12+\eta$, and if $\ell$ is an integer such that $1+2\eta > \left(1 + \frac{1}{2\ell+1}\right)^2$ then for every admissible $k$-tuple, with $k=(2\ell+1)^2$, there are infinitely many $n$ for which the $k$-tuple, evaluated at $n$, contains (at least) two primes.

 \section{Zhang's modifications of GPY}
At the end of the previous section we saw that if the
Elliott-Halberstam conjecture holds with any exponent $> \frac 12$, then  for every admissible $k$-tuple (with $k$ sufficiently large),
there are infinitely many $n$ for which the $k$-tuple contains two primes. However the Elliott-Halberstam conjecture remains unproven.

In \eqref{EHsmooth} we stated Zhang's result, which breaks the $\sqrt{x}$-barrier (in such results), but at
the cost of restricting the moduli to being $y$-smooth, and restricting the arithmetic progressions $a\pmod q$ to having the same value of $a$ as we vary over $q$ .
Can the Goldston-Pintz-Y{\i}ld{\i}r{\i}m argument
be modified to handle these restrictions?

\subsection{Averaging over   arithmetic progressions} In the GPY argument we need estimates for the number of primes in the arithmetic progressions $m+a_i \pmod D$ where $m\in \Omega_i(D)$.  When using the Bombieri-Vinogradov Theorem, it does not matter that $m+a_i$ varies as we vary over $D$; but it does matter when employing Zhang's Theorem \ref{Zhangthm}.

Zhang realized that one can exploit the structure of the sets $O_i(D)=\Omega_i(D)+a_i$, since they are constructed from the sets $O_i(p)$, for each prime $p$ dividing $D$, using the Chinese Remainder Theorem, to get around this issue:

Let $\nu(D)$ denote the number of prime factors of (squarefree) $D$, so that $\tau(D)=2^{\nu(D)}$.
Any squarefree  $D$ can be written as $[d_1,d_2]$ for   $3^{\nu(D)}$ pairs $d_1,d_2$, which means that we need an appropriate upper bound  on
\[
\leq    \sump_{D\leq Q}  3^{\nu(D)}
\sum_{b\in O_i(D)}  \left| \Theta(X;D,b)- \frac X{\phi(D)}\right|
\]
where $Q=R^2$ and $X=x$ or $2x$, for each $i$.

Let $L$ be the lcm of all of the $D$ in our sum.   The set, $O_i(L)$, reduced mod $D$, gives $|O_i(L)|/|O_i(D)|$ copies of $O_i(D)$ and so
\[
\frac 1{|O_i(D)|}
\sum_{b\in O_i(D)}  \left| \Theta(X;D,b)- \frac X{\phi(D)}\right|
= \frac 1{|O_i(L)|}
\sum_{b\in O_i(L)}  \left| \Theta(X;D,b)- \frac X{\phi(D)}\right| .
\]
Now  $|O_i(D)|=\omega^*(D)\leq (k-1)^{\nu(D)}$,   and so
$3^{\nu(D)}|O_i(D)|\leq \tau(D)^A$ for all squarefree $D$, where $A$ is chosen so that $2^A=3(k-1)$.
The above is therefore
 \begin{align*}
&\leq    \sump_{D\leq Q}  \tau(D)^A \cdot \frac 1{|O_i(D)|}
\sum_{b\in O_i(D)}  \left| \Theta(X;D,b)- \frac X{\phi(D)}\right| \\
&=  \frac 1{|O_i(L)|} \sum_{b\in O_i(L)}   \sump_{D\leq Q}  \tau(D)^A \cdot  \left| \Theta(X;D,b)- \frac X{\phi(D)}\right|\\
&\leq \max_{a\in \mathbb Z} \sump_{\substack{D\leq Q \\ (D,a)=1}}  \tau(D)^A \cdot  \left| \Theta(X;D,a)- \frac X{\phi(D)}\right| .
\end{align*}
This can be bounded using Theorem \ref{Zhangthm}, via  a standard technical argument: By Cauchy's Theorem, the square of this is
\[
\leq \sum_{D\leq Q}  \frac{\tau(D)^{2A}}{D} \cdot  \sump_{D\leq Q}    D \left| \Theta(X;D,b)- \frac X{\phi(D)}\right|^2.
\]
The first sum is $\leq \prod_{p\leq Q} (1+\frac 1p)^B$ where $B=2^{2A}$, which is  $\leq (c\log Q)^{B}$ for some constant $c>0$, by Mertens' Theorem. For the second sum,  $D \ \Theta(X;D,b) \leq (X+D)\log X$, trivially, and
$\frac D{\phi(D)}\leq \log X$ (again by Mertens' Theorem), so the second sum is
\[ \leq 2 X\log X \sump_{D\leq Q}      \left| \Theta(X;D,b)- \frac X{\phi(D)}\right|.\]
The sum in this equation may be bounded
by Theorem \ref{Zhangthm}.

\subsection{Restricting the support to smooth integers}
Zhang simply took the same coefficients $y(r)$ as above, but now
restricted to $x^\delta$-smooth integers; and   called this restricted class of coefficients, $z(r)$.  Evidently
the sum in \eqref{squaresum} with $z(r)$ in place of $y(r)$, is bounded above by the sum in
\eqref{squaresum}.  The sum in \eqref{solve2}  with $z(r)$ in place of $y(r)$, is a little more tricky,
since we need a lower bound. Zhang proceeds by showing that if $L$ is sufficiently large and $\delta$
sufficiently small, then the two sums differ by only a negligible amount.\footnote{Unbeknownst to Zhang, Motohashi and Pintz \cite{mp} had already given an argument to accomplish the goals of this section, in the hope that someone might prove an estimate like \eqref{EHsmooth}!} In particular we will prove
Zhang's Theorem when \[ 162\eta+90\delta<1. \] Zhang's argument to restrict the support to smooth integers, as just discussed in this subsection,   holds when $\ell  =431,\ k=(2\ell+1)^2$ and
$\eta=2/(2\ell+1)$.

\section{Goldston-Pintz-Y{\i}ld{\i}r{\i}m in higher dimensional analysis}\label{GPYhigh}

In the GPY argument, we studied the divisors
$d$ of the product of the $k$-tuple values; that is
\[
d|\mathcal P(n) = (n+a_1)\ldots (n+a_k).
\]
with $d\leq R$.

Maynard and Tao (independently) realized that they could  instead study the $k$-tuples of divisors $d_1,d_2,\ldots, d_k$ of each
individual element of the $k$-tuple; that is
\[
d_1|n+a_1, \ d_2|n+a_2, \ldots, d_k|n+a_k.
\]
Now, instead of $d\leq R$, we take $d_1d_2\ldots d_k\leq R$.

\subsection{The set up}
One can proceed much as in the previous section, though technically it is easier to restrict our attention to when
$n$ is in an appropriate congruence class mod $m$ where $m$ is the product of the primes for which $\omega(p)<k$,
because, if $\omega(p)=k$ then $p$ can only divide one $n+a_i$ at a time. Therefore we study
\[
S_0:=\sum_{r\in \Omega(m)} \sum_{\substack{ n\sim x \\ n\equiv r \pmod m}}
\left( \sum_{j=1}^k \theta(n+a_j) - h \log 3x\right)
\left(
\sum_{ d_i|n+a_i \ \text{for each} \ i } \lambda(d_1,\ldots, d_k)
\right)^2
\]
which upon expanding, as $(d_i,m)|(n+a_i,m)=1$, equals
\[
\sum_{\substack{ d_1,\ldots, d_k \geq 1 \\ e_1,\ldots, e_k \geq 1\\ (d_ie_i,m)=1\ \text{for each} \ i }}
\lambda(d_1,\ldots, d_k)\lambda(e_1,\ldots, e_k)
\sum_{r\in \Omega(m)}
   \sum_{\substack{ n\sim x \\ n\equiv r \pmod m \\ [d_i,e_i] | n+a_i\text{for each} \ i \ }}
\left( \sum_{j=1}^k \theta(n+a_j) - h \log 3x\right)  .
\]
Next notice that $[d_i,e_i]$ is coprime with $[d_j,e_j]$ whenever $i\ne j$, since their gcd divides
$(n+a_j)-(n+a_i)$, which divides $m$, and so equals $1$ as $(d_ie_i,m)=1$. Hence, in our internal sum, the values of $n$ belong to an arithmetic progression with modulus  $m\prod_i [d_i,e_i]$. Also notice that if $n+a_j$ is prime then $d_j=e_j=1$.

Therefore,
ignoring error terms,
\[
S_0=   \sum_{\substack{1\leq \ell\leq k}} \frac{\omega(m) } {\phi(m)}  S_{2,\ell} \cdot x
-h \frac{\omega(m) } m S_1 \cdot x\log 3x
\]
where
\[
S_1:= \sum_{\substack{ d_1,\ldots, d_k \geq 1 \\ e_1,\ldots, e_k \geq 1 \\
(d_i,e_j)=1 \ \text{for} \ i\ne j   }}
\frac{\lambda(d_1,\ldots, d_k)\lambda(e_1,\ldots, e_k)  }{\prod_i\ [d_i,e_i]}
\]
and
\[
S_{2,\ell}:= \sum_{\substack{ d_1,\ldots, d_k \geq 1 \\ e_1,\ldots, e_k \geq 1\\
(d_i,e_j)=1 \ \text{for} \ i\ne j \\ d_\ell=e_\ell=1  }}
\frac{\lambda(d_1,\ldots, d_k)\lambda(e_1,\ldots, e_k)  }{\prod_i\phi( [d_i,e_i])}.
\]
These sums can be evaluated in several ways. Tao (see his blog) gave what is  perhaps the simplest approach, generalizing the Fourier analysis technique discussed in section \ref{sect:TaoApproach} (see \cite{polymath8b}, Lemma 4.1). Maynard \cite{maynard} gave what is, to my taste, the more elegant approach, generalizing Selberg's combinatorial technique:

\subsection{The combinatorics} The reciprocity law generalizes quite beautifully to higher dimension:\ Suppose that $L(d)$ and $Y(r)$ are two sequences of complex numbers, indexed by $d,r\in \mathbb Z_{\geq 1}^k$, and non-zero only when each $d_i$ (or $r_i$) is  squarefree. Then
 \[
 L(d_1,\ldots, d_k) = \prod_{i=1}^k  \mu(d_i)  \sum_{\substack{ r_1,\ldots , r_k\geq 1 \\ d_i|r_i \ \text{for all } i}} Y(r_1,\ldots , r_k)
 \]
 if and only if
  \[
Y(r_1,\ldots , r_k)  = \prod_{i=1}^k  \mu(r_i)  \sum_{\substack{ d_1,\ldots, d_k \geq 1 \\ r_i|d_i \ \text{for all } i}}
 L(d_1,\ldots, d_k) .
  \]
We use this much as above, in the first instance with
 \[
L(d_1,\ldots, d_k) = \frac{\lambda(d_1,\ldots, d_k) } {d_1,\ldots, d_k}  \ \text{   and    }
Y(r_1,\ldots , r_k)  = \frac{y(r_1,\ldots, r_k) } {\phi_k(r_1\ldots r_k)}
 \]
where
\[
y(r_1,\ldots, r_k) =   F\left(  \frac{\log r_1}{\log R},\ldots,  \frac{\log r_k}{\log R} \right)
\]
with $F\in \mathbb C[t_1,\ldots, t_k]$, such that that there is a uniform bound on all of the first order partial derivatives, and $F$ is only supported on
\[
T_k:=\{ (t_1,\ldots, t_k):\ \text{Each} \ t_j\geq 1, \text{ and} \ t_1+\ldots+ t_k\leq 1\}.
 \]
 Proceeding much as before we obtain
\begin{equation} \label{S1value}
S_1  \sim  \sum_{ r_1,\ldots , r_k\geq 1  } \frac{y(r_1,\ldots , r_k) ^2}{\phi_k(r_1\ldots  r_k)} .
\end{equation}

\subsection{Sums of multiplicative functions}  By \eqref{SD} we have
\begin{equation}\label{sumeval}
 \sum_{\substack{ 1\leq n\leq N \\  (n,m )=1} }   \frac{\mu^2(n)}{\phi_k(n)} =
 \prod_{p|m}  \frac {p-1}{p}  \prod_{p\nmid m}      \frac{(p-1)\phi_{k-1}(p)}{p\ \phi_k(p)} \cdot (\log N+O(1))
\end{equation}
We apply  this $k$ times; firstly with $m$ replaced by $mr_1\ldots r_{k-1}$ and $n$ by $r_k$, then with $m$ replaced by $mr_1\ldots r_{k-2}$, etc By the end we obtain
\begin{equation} \label{step0}
C_m(a) \sum_{\substack{1\leq  r_1\leq R_1, \\ \ldots , \\ 1\leq  r_k\leq R_k} } \frac{\mu^2(r_1\ldots r_k m)}{\phi_k(r_1,\ldots , r_k)}  \ =    \prod_i (\log R_i+O(1)) ,
\end{equation}
where
\[
C_m(a):=\ \prod_{p| m}   \left( 1 -\frac{ 1 } {p }\right) ^{-k}
 \prod_{p\nmid  m}  \left( 1 -\frac{ k } {p }\right)   \left( 1 -\frac{ 1 } {p }\right)^{-k} .
\]
From this, and partial summation, we deduce from \eqref{S1value}, that
\begin{equation} \label{ResultS1}
C_m(a) S_1 \sim
(\log R)^k  \cdot  \int_{t_1,\ldots, t_k\in T_k}  F(t_1,\ldots, t_k)^2 dt_k\ldots dt_1 .
\end{equation}

Had we stopped our calculation one step earlier we would have found
\begin{equation} \label{step1}
 C_m(a) \sum_{\substack{1\leq  r_1\leq R_1, \\ \ldots , \\ 1\leq  r_{k-1}\leq R_{k-1}} } \frac{\mu^2(r_1\ldots r_{k-1} m)}{\phi_k(r_1,\ldots , r_{k-1})}  \ = \ \frac m{\phi(m)}   \cdot  \prod_i (\log R_i+O(1)) ,
\end{equation}

\begin{remark} If we replace $\lambda(d)\lambda(e)$ by $\lambda_1(d)\lambda_2(e)$ in the definition of $S_1$,
then the analogous argument yields  $\int_{t\in T_k}  F(t)^2 dt$ replaced by $\int_{t\in T_k}  F_1(t)F_2(t) dt$ in \eqref{ResultS1}.
\end{remark}

\subsection{The combinatorics, II}  We will deal only with the case $\ell=k$, the other cases being analogous. Now we use the higher dimensional reciprocity law with
\[
L(d_1,\ldots, d_{k-1}) = \frac{\lambda(d_1,\ldots, d_{k-1},1) } {\phi( d_1\ldots d_{k-1}) }  \ \text{   and    }
Y_k(r_1,\ldots , r_{k-1})  = \frac{y_k(r_1,\ldots, r_{k-1}) } {\phi_k(r_1\ldots r_{k-1})}
 \]
where $d_k=r_k=1$, so that, with the exactly analogous calculations as before,
\[
S_{2,k} \sim \sum_{ r_1,\ldots , r_{k-1}\geq 1  } \frac{y_k(r_1,\ldots , r_{k-1}) ^2}{\phi_k(r_1\ldots  r_{k-1})} .
\]
Using the reciprocity law twice to determine the $y_k(r)$ in terms of the $y(n)$, we obtain that
\[
y_k(r_1,\ldots , r_{k-1})\sim  \frac{\phi(m)}{m} \cdot
       \int_{t\geq 0} F(\rho_1,\ldots, \rho_{k-1}, t) dt\cdot   \log R
\]
where each $r_i=N^{\rho_i}$. Therefore, using \eqref{step1}, we obtain
\begin{equation} \label{ResultS2}
C_m(a)S_{2,k}\sim
  \  \int_{0\leq t_1,\ldots, t_{k-1}\leq 1}  \left( \int_{t_k\geq 0} F(t_1,\ldots, t_{k-1}, t_k) dt_k\right)^2 dt_{k-1}\ldots dt_1   \cdot \frac{\phi(m)}{m}    (\log R)^{k+1} .
\end{equation}

\begin{remark} If we replace $\lambda(d)\lambda(e)$ by $\lambda_1(d)\lambda_2(e)$ in the definition of $S_2$,
then the analogous argument yields  $(\int_{t_k\geq 0}  F(t) dt_k)^2$ replaced by $(\int_{t_k\geq 0}  F_1(t) dt_k)(\int_{t_k\geq 0}  F_2(t) dt_k)$ in \eqref{ResultS2}.
\end{remark}

\subsection{Finding a positive difference} By the Bombieri-Vinogradov Theorem we can take $R=x^{1/4-o(1)}$, so that, by \eqref{ResultS1} and  \eqref{ResultS2},  $C_m(a) S_0$ equals
$\frac{\omega(m) } {m}  x(\log 3x) (\log R)^k$ times
\[
\frac 14 \sum_{\ell=1}^{k}  \int_{\substack{0\leq t_i\leq 1 \ \text{for} \\ 1\leq i\leq k,\ \ i\ne \ell }}  \left( \int_{t_\ell\geq 0} F(t_1,\ldots, t_k) dt_\ell\right)^2 \prod_{\substack{1\leq j\leq k\\  i\ne \ell}}  dt_j     -  h  \int_{t_1,\ldots, t_k\in T_k}  F(t_1,\ldots, t_k)^2 dt_k\ldots dt_1 +o(1) .
\]
One can show that the optimal choice for  $F$ must be symmetric. Hence  $S_0>0$  follows if there exists a symmetric $F$ (with  the restrictions above) for which   the ratio
\[
\rho(F):=\frac{k \int_{t_1,\ldots,t_{k-1}\geq 0}  \left( \int_{t_k\geq 0} F(t_1,\ldots, t_k) dt_k\right)^2dt_{k-1} \ldots dt_1   }{\int_{t_1,\ldots, t_k\geq 0}  F(t_1,\ldots, t_k)^2 dt_k\ldots dt_1 } .
\]
satisfies $\rho(F)>4h$.

We have proved the following Proposition which leaves us to find functions $F$ with certain properties, in order to obtain the main results:

\begin{proposition} \label{TechProp} Fix $h\geq 1$.  Suppose that there exists $F\in \mathbb C(x_1,\ldots,x_k)$ which is measurable, supported on $T_k$, for which there is a uniform bound on the first order partial derivatives and such that
 $\rho(F)>4h$. Then,   for every admissible $k$-tuple of linear forms, there are infinitely many integers $n$ such that there are $>h$ primes amongst the $k$ linear forms when evaluated at $n$. If  the Elliott-Halberstam conjecture holds then we  only need  that $\rho(F)>2h$.
\end{proposition}

\subsection{A special case} If $F(t_1,\ldots,t_k)=f(t_1+\ldots+t_k)$ then since
\[
\int_{\substack{t_1,\ldots, t_k\geq 0 \\ t_1+\ldots+t_k=t}}  dt_{k-1} \ldots dt_1 = \frac{t^{k-1}}{(k-1)!} ,
\]
we deduce that
\[
\rho(F)=\rho_k(f)
\]
 as defined in \eqref{Rhok}. That is, we have reverted to the original GPY argument, which was not quite powerful enough for our needs. We want to select $F$ that does not lead us back to the original GPY argument, so we should avoid selecting $F$ to be a function of one variable.

 Since $F$ is symmetric, we can define the symmetric sums,  $P_j=\sum_{i=1}^k t_i^j$. In the GPY argument $F$ was a function of $P_1$. A first guess might be to work now with functions of $P_1$ and $P_2$, so as to consider functions $F$ that do not appear in the GPY argument.

\subsection{Maynard's $F$s, and gaps between primes}    For $k=5$ let
\[
F(t_1,\ldots, t_5) = 70P_1P_2 - 49P_1^2 - 75P_2 + 83P_1-34 .
\]
A  calculation yields that
\[
\rho(F)=\frac{1417255}{708216}>2.
\]
Therefore, by Proposition \ref{TechProp},  if we assume the Elliott-Halberstam conjecture with $h=1$ then for every admissible $5$-tuple of linear forms, there are infinitely many integers $n$ such that there are at least two primes amongst the five linear forms when evaluated at $n$. In particular, from the admissible forms $\{ x, x+2, x+6, x+8, x+12\}$ we deduce that there are infinitely many pairs of distinct primes that differ by no more than $12$.  Also from the admissible forms $\{ x+1, 2x+1, 4x+1, 8x+1, 16x+1\}$ we deduce that there are infinitely many pairs of distinct primes, $p,q$ for which $(p-1)/(q-1)=2^j$ for $j=0,1,2,3$ or $4$.

Maynard \cite{maynard} showed that there exists a polynomial of the form
\[
\sum_{\substack{a,b\geq 0 \\ a+2b\leq 11}} c_{a,b} (1-P_1)^a P_2^b
\]
with $k=105$, for which
\[
\rho(F)=4.0020697 \ldots
\]
By Proposition \ref{TechProp}  with $h=1$, we can then deduce that for every admissible $105$-tuple of linear forms, there are infinitely many integers $n$ such that there are at least two primes amongst the $105$ linear forms when evaluated at $n$.

How did Maynard find his polynomial, $F$, of the above form? The numerator and denominator of $\rho(F)$ are quadratic forms in the 42 variables $c_{a,b}$, say $v^TM_2v$ and $v^TM_1v$, respectively, where $v$ is the vector of $c$-values. The matrices $M_1$ and $M_2$ are easily determined.    By the theory of Lagrangian multipliers, Maynard showed that
\[
M_1^{-1}M_2 v = \rho(F) v
\]
so that $\rho(F)$ is the largest eigenvalue of $M_1^{-1}M_2$, and $v$ is the corresponding eigenvector. These calculations are easily accomplished using a computer algebra package and yield the result above.

\subsection{$F$ as a product of one dimensional functions}  We  select
 \[
 F(t_1,\ldots t_k) = \begin{cases}  g(kt_1)\ldots g(kt_k) & \text{if} \ t_1+\ldots + t_k\leq 1 \\
 0 &\text{otherwise}, \end{cases}
 \]
where $g$ is some integrable function supported only on $[0,T]$. Let   $\gamma:=\int_{t\geq 0} g(t)^2dt$,
so that the denominator of $\rho(F)$ is
\[
I_k= \int_{t\in T_k} f(t_1,\ldots t_k)^2 dt_k\ldots dt_1 \leq  \int_{t_1,\ldots,t_k\geq 0} (g(kt_1)\ldots g(kt_k) )^2 dt_k\ldots dt_1 = k^{-k}\gamma^k.
 \]
We rewrite the numerator of $\rho(F)$ as
$L_k-M_k$ where
 \[
L_k:= k \int_{t_1,\ldots, t_{k-1}\geq 0} \left( \int_{t_k\geq 0} g(kt_1)\ldots g(kt_k) dt_k\right)^2 dt_{k-1}\ldots dt_1
 =  k^{-k} \gamma^{k-1} \left(  \int_{t\geq 0} g(t) dt\right)^2   .
\]
As $g(t)$ is only supported in $[0,T]$ we have, by Cauchying and letting $u_j=kt_j$,
\begin{align*}
M_k:&=  \int_{t_1,\ldots, t_{k-1}\geq 0} \left( \int_{t_k\geq 1-t_1-\ldots-t_{k-1}} g(kt_1)\ldots g(kt_k) dt_k\right)^2 dt_{k-1}\ldots dt_1 \\
& \leq k^{-k} T \int_{\substack{u_1,\ldots, u_k\geq 0 \\ u_1+\ldots+ u_k\geq k}}  g(u_1)^2\ldots g(u_k)^2 du_1\ldots du_k.
\end{align*}
Now assume that $\mu:=\int_t tg(t)^2dt\leq  (1-\eta) \int_t g(t)^2dt=(1-\eta) \gamma$ for some given $\eta>0$;  that is, that the ``weight'' of $g^2$ is centered around values of $t\leq 1-\eta$. We have
\[
1\leq \eta^{-2}\left( \frac 1k (u_1+\ldots+u_k) - \mu/\gamma\right)^2
\]
 whenever $u_1+\ldots+ u_k\geq k$. Therefore,
\begin{align*}
M_k &\leq \eta^{-2} k^{-k} T \int_{u_1,\ldots, u_k\geq 0}  g(u_1)^2\ldots g(u_k)^2 \left( \frac 1k (u_1+\ldots+u_k) - \mu/\gamma\right)^2 du_1\ldots du_k\\
&= \eta^{-2} k^{-k-1} T \int_{u_1,\ldots, u_k\geq 0}  g(u_1)^2\ldots g(u_k)^2 ( u_1^2 - \mu^2/\gamma^2) du_1\ldots du_k\\
&= \eta^{-2} k^{-k-1}  \gamma^{k-1}T \left( \int_{u\geq 0}  u^2g(u)^2du- \mu^2/\gamma\right)
\leq  \eta^{-2} k^{-k-1}  \gamma^{k-1}T   \int_{u\geq 0}  u^2g(u)^2du,
\end{align*}
by symmetry. We deduce that
 \begin{equation} \label{1stLowerBound}
 \rho(F) \geq \frac { \left(  \int_{t\geq 0} g(t) dt\right)^2 - \frac{\eta^{-2} T}k \int_{u\geq 0}  u^2g(u)^2du}
 {  \int_{t\geq 0} g(t)^2dt  } .
\end{equation}
Notice that we can multiply $g$ through by a scalar and not affect the value in \eqref{1stLowerBound}.

\subsection{The optimal choice} We wish to find the value of $g$ that maximizes the right-hand side of \eqref{1stLowerBound}. This can be viewed as an optimization problem:
\smallskip

\emph{Maximize} $\int_{t\geq 0} g(t) dt$, \ subject to the constraints
$\int_{t\geq 0} g(t)^2 dt=\gamma $ and $\int_{t\geq 0} tg(t)^2 dt=\mu $.
\smallskip

One can approach this using the calculus of variations or even by discretizing $g$ and employing the technique of  Lagrangian multipliers. The latter gives rise to (a discrete form of)
\[
\int_{t\geq 0} g(t) dt -\alpha \left( \int_{t\geq 0} g(t)^2 dt-\gamma\right) -\beta \left( \int_{t\geq 0} tg(t)^2 dt-\mu\right) ,
\]
for unknowns $\alpha$ and $\beta$. Differentiating with respect to $g(v)$ for each $v\in [0,T]$, we obtain
\[ 1 - 2\alpha g(v) - 2\beta v g(v) = 0;\]
that is, after re-scaling,
\[
g(t)=\frac 1 {1+At}\ \ \text{for} \ \ 0\leq t\leq T,
\]
for some real $A>0$. We select $T$ so that $1+AT=e^A$, and let $A>1$.
We then calculate the   integrals in \eqref{1stLowerBound}:
\begin{align*}
 \gamma=\int_t g(t)^2dt \  &= \frac 1A (1-e^{-A}),\\
 \int_t tg(t)^2dt &= \frac1{A^2}\left( A -1+e^{-A}\right),\\
 \int_t t^2g(t)^2dt &=  \frac1{A^3}\left( e^{A}-2 A - e^{-A} \right), \\
 \text{and} \qquad \qquad  \int_t g(t)dt \ \ \  &= 1, \\
  \text{so that} \qquad \qquad \qquad  \eta\ \ \ &=   \frac{1-(A-1)e^{-A}} {A(1-e^{-A})} \ >0,
 \end{align*}
which is necessary. \eqref{1stLowerBound} then becomes
 \begin{equation} \label{2ndLowerBound}
 \rho(F) \geq   \frac A { (1-e^{-A}) } -
       \frac{e^{2A}}{Ak}\left( 1-2 Ae^{-A} - e^{-2A} \right)  \frac{  (1-e^{-A})^2 } {   (1-(A-1)e^{-A})^2}\geq
      A -  \frac{e^{2A}}{Ak}
\end{equation}
Taking  $A=\frac 12 \log k+\frac 12 \log\log k$, we deduce that
 $$\rho(F) \geq \frac 12 \log k+\frac 12 \log\log k-2.$$
 Hence, for every $m\geq 1$ we find that
$\rho(F)>4m$ provided $ e^{8m+4}<  k \log k$.

This implies the following result:

\begin{theorem} For any given integer $m\geq 2$, let $k$ be the smallest integer with $k \log k>e^{8m+4}$.  For \emph{any} admissible $k$-tuple of linear forms $L_1,\ldots, L_k$ there exists infinitely many integers $n$ such that at least $m$ of the $L_j(n),\ 1\leq j\leq k$ are prime.
\end{theorem}

For any $m\geq 1$, we let $k$ be the  smallest integer with $k \log k>e^{8m+4}$, so that $k>10000$; in this range it is known that $\pi(k)\leq \frac k{\log k-4}$.  Next we let $x=2k\log k >10^5$ and, for this range it is known that $\pi(x)\geq \frac x{\log x}( 1 + \frac 1 {\log x} )$. Hence
\[
\pi(2k\log k)-\pi(k) \geq \frac {2k\log k}{\log (2k\log k)}\left( 1 + \frac 1 {\log (2k\log k)} \right) - \frac k{\log k-4}
\]
and this is $>k$ for $k\geq 311$ by an easy calculation. We therefore apply the theorem with the $k$ smallest primes $>k$, which form an admissible set $\subset [1,2k\log k]$, to obtain:

\begin{corollary} For any given integer $m\geq 2$, let $B_m=e^{8m+5}$. There are  infinitely many integers $x$ for which there are at least $m$ distinct primes within the interval $[x,x+B_m]$.
\end{corollary}

By a slight modification of this construction, Maynard in \cite{maynard} obtains
\begin{equation}\label{Maybound}
\rho(F) \geq  \log k-2 \log\log k-1+o_{k\to\infty}(1)
\end{equation}
from which he analogously deduces that
$B_m\ll m^3 e^{4m}$.

By employing Zhang's improvement, Theorem \ref{Zhangthm}, to the Bombieri-Vinogradov Theorem, one can improve this  to $B_m\ll m e^{(4-\frac{28}{157})m}$ and if we assume the Elliott-Halberstam conjecture, to
$B_m\ll e^{(2+o(1))m}$. The key to significantly  improving these upper bounds on $B_m$ is to obtain a much better lower bound on $\rho(F)$.

\subsection{Tao's upper bound on $\rho(F)$}   By the Cauchy-Schwarz inequality we have
\[
 \left( \int_{t_k=0}^{1-t_1-\ldots -t_{k-1}} F(t_1,\ldots, t_k) dt_k\right)^2
  \leq \int_{t_k=0}^{1-t_1-\ldots -t_{k-1}}  \frac{dt_k}{1-t_1-\ldots -t_{k-1}+(k-1)t_k} \cdot
  \]
 \[
\times
 \int_{t_k=0}^{1-t_1-\ldots -t_{k-1}}  F(t_1,\ldots, t_k)^2 (1-t_1-\ldots -t_{k-1}+(k-1)t_k) dt_k.
 \]
Letting $u=A+(k-1)t$ we have
 \[
 \int_0^A \frac{dt}{A+(k-1)t} = \frac 1{k-1} \int_A^{kA} \frac{du}{u}  = \frac {\log k}{k-1} .
 \]
Applying this with $A=1-t_1-\ldots -t_{k-1}$ and $t=t_k$, yields an upper bound for the numerator of $\rho(F)$:
 \begin{align*}
& \leq
  \frac {\log k}{k-1}  \cdot  \int_{t_1,\ldots,t_{k}\geq 0}   F(t_1,\ldots, t_k)^2\sum_{j=1}^k (1-t_1-\ldots -t_{k}+kt_j) dt_k\ldots dt_1
 \\ &=
  \frac {k\log k}{k-1}  \cdot  \int_{t_1,\ldots,t_{k}\geq 0}   F(t_1,\ldots, t_k)^2  dt_k\ldots dt_1 .
\end{align*}
  Hence
\[
\rho(F) \leq   \frac {k\log k}{k-1}  = \log k +o_{k\to\infty}(1)
\]
 so there is little room for improvement of Maynard's upper bound, \eqref{Maybound}.


\part{Primes in arithmetic progressions; breaking the $\sqrt{x}$-barrier}
Our goal, in the rest of the article, is to sketch the ideas behind a proof of Yitang Zhang's extraordinary result, given in \eqref{EHsmooth}, that primes are well-distributed on average in the  arithmetic progressions $a \pmod q$ with $q$ a little bigger than $\sqrt{x}$.
We will see how this question fits into a more general framework, as developed by Bombieri, Friedlander and Iwaniec \cite{bfi}, so that Zhang's results  should also allow us to deduce analogous results for interesting arithmetic sequences other than the primes.

For the original, much deeper and more complicated proof of Yitang Zhang, the reader is referred to the insightful exposition by Kowalski \cite{Kowalski}, which motivates and develops these difficult ideas with great clarity. One can also read a slightly different development of Zhang's theorem by Friedlander and Iwaniec \cite{fi-4}, which incorporates various novel features.

To begin our discussion, we will introduce a key technique of analytic number theory, the idea of creating important sequences through convolutions:

\section{Convolutions in number theory}

The convolution of two functions $f$ and $g$, written  $f*g$, is defined by
\[
(f*g)(n) :=  \sum_{ab=n}  f(a)g(b),
\]
for every integer $n\geq 1$, where the sum is over all pairs of positive integers $a,b$ whose product is $n$. Hence if $\tau(n)$ counts the number of divisors of $n$ then
\[\tau=1*1,\] where $1$ is the function with $1(n)=1$ for every $n\geq 1$. We already saw, in \eqref{VMidentity}, that if  $L(n)=\log n$ then $\mu*L=\Lambda$.  In the GPY argument we used that $(1*\mu)(n)=0$ if $n>1$.

\subsection{Dirichlet's divisor sum} There is no better way to understand why convolutions are useful than to present a famous argument of Dirichlet, estimating the average of $\tau(n)$: If $n$ is squarefree and has $k$ prime factors then $\tau(n)=2^k$, so we see that $\tau(n)$ varies greatly depending on the arithmetic structure of $n$, but the average is more stable:
 \begin{align*}
 \frac 1x\sum_{n\leq x} \tau(n) &= \frac 1x \sum_{n\leq x} \sum_{d|n} 1 = \frac 1x \sum_{d|n} \sum_{\substack{n\leq x \\ d|n}} 1 =
 \frac 1x\sum_{d\leq x} \left[ \frac xd \right] \\
 &=  \frac 1x  \sum_{d\leq x} \left( \frac xd +O(1) \right)
 =  \sum_{d\leq x}  \frac 1d +O\left(     \frac 1x\sum_{d\leq x} 1 \right) .
 \end{align*}
One can approximate $\sum_{d\leq x}  \frac 1d$ by $\int_1^x dt/t =\log x$. Indeed  the difference tends to a limit, the Euler-Mascheroni constant $\gamma:= \lim_{N\to \infty} \frac 11+\frac 12 +\ldots +\frac 1N -\log N$. Hence we have proved  that the integers up to $x$ have $\log x+O(1)$ divisors, on average, which is quite remarkable for such a wildly fluctuating function.

Dirichlet studied this argument and noticed that when we approximate $[x/d]$ by $x/d+O(1)$ for large $d$, say for those $d$ in $(x/2,x]$, then this is not really a very good approximation, since $\sum_{x/2<d\leq x} ( x/d+O(1))=O(x)$,
a large cumulative error term. However we have the much more precise   $[x/d]=1$   for each of these $d$, and so we can estimate this sum by $x/2+O(1)$, a much better approximation. In general we write $n=dm$, where $d$ and $m$ are integers. When $d$ is small then we should fix $d$, and count the number of such $m$, with $m\leq x/d$ (as we did above); but when $m$ is small, then we should fix $m$, and count the number of  $d$ with $d\leq x/m$.  In this way our sums are all over long intervals, which allows us to get an accurate approximation of their value:
\begin{align*}
\frac 1x  \sum_{n\leq x} \tau(n) &= \frac 1x \sum_{n\leq x} \sum_{dm=n} 1 =
\frac 1x \sum_{d\leq \sqrt{x} }   \sum_{\substack{n\leq x \\ d|n}} 1+   \frac 1x\sum_{m< \sqrt{x} }   \sum_{\substack{n\leq x \\ m|n}} 1
 - \frac 1x \sum_{d\leq \sqrt{x} }    \sum_{m< \sqrt{x} }  1  \\
 &=
\frac 1x  \sum_{d\leq \sqrt{x} }   \left( \frac xd +O(1) \right)+  \frac 1x \sum_{m< \sqrt{x} }  \left( \frac xm +O(1) \right)  - 1 +O\left( \frac 1{\sqrt{x}}\right) \\
 &=    \log x +2\gamma -1 +O\left( \frac 1{\sqrt{x}}\right) ,
 \end{align*}
since $\sum_{n\leq N} 1/n = \log N+\gamma +O(1/N)$, an extraordinary improvement upon the earlier error term.

\subsection{Vaughan's identity} If we sum  \eqref{VMidentity} over all $n\leq x$ and trivially bound the contribution of the prime powers, then we obtain
\[
\Theta(x) = \sum_{ab\leq x} \mu(a)\log b +O(\sqrt{x}).
\]
If we fix $a$, and then sum over $b\leq B$, where $B=[x/a]$, we obtain $\mu(a) \log B!$ and we can approximate $B!$ very well using Stirling's formula. Hence the key difficulty in using this to approximate $\Theta(x)$ is to understand the sum  of $\mu(a)$ times a smoothish function, for all integers $a\leq x$.  We already discussed this a little in section \ref{pntMob} and the problem remains that $\mu(.)$ is not an easy function to sum.

It is not difficult to find more complicated generalizations of our identity  $\Lambda=\mu*\log$, but to what end? Vinogradov made the extraordinary observation that, in certain ranges, it is possible to give good bounds on such convolutions, ignoring the precise details of the arithmetic function involved but rather getting bounds in terms of certain simpler sums involving the absolute value of those functions. The key is the bilinear shape of the convolutions.   Vinogradov's strategy lies at the heart of all of the proofs in this area.

There are several suitable, more convoluted identities than \eqref{VMidentity}, in which  $\Lambda(n)$ is written as a linear combination of convolutions of arithmetic functions.  The simplest is Vaughan's identity \cite{vaughan}, and will suffice for our needs:
\begin{equation} \label{Vaughidentity} \emph{Vaughan's identity}: \qquad \Lambda_{\geq V} = \mu_{<U} * L  - \mu_{<U} * \Lambda_{< V} * 1 + \mu_{\geq U} * \Lambda_{\geq V} * 1
\end{equation}
where $g_{>W}(n)=g(n)$ if $n>W$ and $g(n)=0$ otherwise; and $g=g_{\leq W}+g_{>W}$.
To verify this identity, we manipulate the algebra of convolutions:
\begin{align*}
 \Lambda_{\geq V} & =  \Lambda- \Lambda_{< V} = (\mu*L)- \Lambda_{< V}*(1*\mu) \\
 &= \mu_{<U}*L+ \mu_{\geq U} *L - \mu_{<U}*\Lambda_{< V}*1 - \mu_{\geq U}*\Lambda_{< V}*1\\
  &= \mu_{<U}*L - \mu_{<U}*\Lambda_{< V}*1 + \mu_{\geq U}*(\Lambda*1-\Lambda_{< V}*1) ,
\end{align*}
 from which we immediately deduce \eqref{Vaughidentity}.

The following identity,   due to Heath-Brown \cite{hb-ident}, is used in \cite{polymath8} to get the strongest form of Zhang's theorem: If $n\leq U^k$ then
\[
\Lambda(n)=-\sum_{j=1}^k (-1)^j \binom kj \mu_{\leq U} * \mu_{\leq U} *\ldots * \mu_{\leq U} * \log
*1*\ldots*1 (n)
\]
where $\mu_{\leq U}$ is convolved $k$ times, and $1$ is convolved $k-1$ times.  This larger number of terms allows us to group divisors in such a way that we have closer control over their sizes.

\section{Distribution in arithmetic progressions} \label{GeneralBV}

\subsection{General sequences in arithmetic progressions} \label{genseqs}
One can ask whether \emph{any} given sequence $(\beta(n))_{n\geq 1}\in \Bbb C$ is well-distributed in arithmetic progressions modulo $q$. We begin by formulating an appropriate analogy to \eqref{SW1}, which should imply non-trivial estimates in the range $q\leq (\log x)^A$ for any fixed $A>0$: We say that $\beta$ satisfies a \emph{Siegel-Walfisz condition} if, for any fixed $A>0$, and whenever $(a,q)=1$, we have
$$
\left| \sum_{\substack{ n\leq x\\ n\equiv a \pmod q}} \beta(n) - \frac
{1}{\phi(q)}\sum_{\substack{  n\leq x\\ (n,q) =1}} \beta(n) \right| \ll_A
   \ \frac{\|\beta\| x^\frac {1}{2}} {(\log x)^A} \ ,
$$
with $\| \beta\| = \| \beta\|_2$ where, as usual,
$$
\|\beta \|_2 :=  \left( \sum_{n\leq x} |\beta(n)|^2\right)^{\frac 12}.
$$
Using  Cauchy's inequality one can show that this assumption  is ``non-trivial'' only for $q<(\log x)^{2A}$; that is, when $x$ is very large compared to $q$.

Using the large sieve, Bombieri, Friedlander and Iwaniec \cite{bfi}  proved two  results that are surprising strong, given the weakness of the hypotheses. In the first they showed that if $\beta$ satisfies a Siegel-Walfisz condition,\footnote{Their condition appears to be weaker than that assumed here, but it can be shown to be equivalent.} then it is well-distributed  for \emph{almost all} arithmetic progressions $a \pmod q$, for \emph{almost all} $q \leq x/(\log x)^B$:

\begin{theorem}  Suppose that the sequence of complex numbers
$\beta(n), n\leq x$ satisfies a Siegel-Walfisz condition. For any
$A>0$ there exists $B=B(A)>0$ such that
$$
\sum_{q\leq Q} \ \sum_{a:\ (a,q)=1} \ \left| \sum_{\substack{ n\equiv a \pmod
q }} \beta(n) - \frac {1}{\phi(q)} \sum_{\substack{ (n,q)=1 }}
\beta(n)\right|^2 \ll \| \beta\|^2 \frac{x} {(\log x)^A}
$$
where $Q=x/(\log x)^B$.
\end{theorem}

The analogous result for $\Lambda(n)$ is known as the \emph{Barban-Davenport-Halberstam theorem} and in that special case one can even obtain an asymptotic.

Before proceeding, let us assume, for the rest of this article, that  we are given two sequences of complex numbers as follows:
\begin{itemize}
\item {$\alpha(m),\ M<m\leq 2M$ and $ \beta(n),\ N<n\leq 2N$, with $x^{1/3}<N\leq M\leq x^{2/3}$ and $MN\leq x$.}

\item {$\beta(n)$ satisfies the Siegel-Walfisz condition.}

\item {$\alpha(m)\ll \tau(m)^A(\log x)^B$ and $\beta(n)\ll \tau(n)^A(\log x)^B$ (these inequalities are satisfied by $\mu, 1, \Lambda, L$ and any convolutions of these sequences).}
\end{itemize}

In their second result, Bombieri, Friedlander and Iwaniec, showed  that rather general convolutions are well-distributed\footnote{This possibility has its roots in  papers of Gallagher \cite{gall} and of Motohashi \cite{mot}.}  for \emph{all} arithmetic progressions $a \pmod q$, for \emph{almost all} $q \leq x^{1/2}/(\log x)^B$.

\begin{theorem} \label{BFI2} Suppose that    $\alpha(m)$ and $\beta(n)$ are as above. For any $A>0$
there exists $B=B(A)>0$ such that
$$
\sum_{q\leq Q} \ \max_{a:\ (a,q)=1} \ \left| \sum_{\substack{ n\equiv a \pmod
q }} (\alpha*\beta)(n) - \frac {1}{\phi(q)} \sum_{\substack{ (n,q)=1 }}
(\alpha*\beta)(n) \right| \ll \| \alpha\| \| \beta\| \frac{x^{1/2}} {(\log x)^A}
$$
where $Q=x^{1/2}/(\log x)^B$. Here we allow any $M$ and $N$ for which $MN\leq x$ and   $x^{\epsilon}\ll M,N \ll   x^{1-\epsilon}$
\end{theorem}

The ranges for $M$ and $N$ are quite restricted in the result, as we have stated it, though the proof is valid in the wider range, $N\geq \exp((\log x)^\epsilon)$ and $M\geq  (\log x)^{2B+4}$. However this is still not a wide enough range to deduce the Bombieri-Vinogradov theorem for primes from \eqref{VMidentity}, but it is wide enough if one uses Vaughan's identity \eqref{Vaughidentity}. This proof of the Bombieri-Vinogradov theorem  follows the pattern laid down by Vinogradov, in  that it seems to be less dependent on   specific properties of the primes.

The restriction on the moduli staying below $x^{1/2}$ has been a major barrier to the development of analytic number theory .\footnote{There had been some partial progress with moduli $>x^{1/2}$, as in \cite{bfi-2}, but no upper bounds that ``win'' by an arbitrary power of $\log x$ (which is what is essential to many applications).}

In the Bombieri-Vinogradov Theorem one bounds the largest error term in the prime number theorem as one varies over all
arithmetic progressions mod $q$, averaging over $q\leq Q$.  In some applications (such as the one here), 
it suffices to bound the error term in the prime number theorem for the arithmetic progression $a \pmod q$, 
averaging over $q$, for one given integer $a$,
as long as we obtain the same bound for all integers $a$.

Bombieri, Friedlander and Iwaniec \cite{bfi} made the following conjecture.\footnote{They actually conjectured that one can take $Q=x/(\log x)^B$.  They also conjectured that if one assumes  the Siegel-Walfisz condition with $\|\beta\|_s N^{1-\frac {1}{s}}$ in place of $\|\beta\| N^\frac {1}{2}$ then we may replace $ \| \alpha\| \| \beta\|  x^{1/2}$ in the upper bound here by $\| \alpha\| M^{1-\frac {1}{r}} \| \beta\| N^{1-\frac 1s}$.}

\begin{conjecture} [Generalized Elliott-Halberstam Conjecture] Suppose that    $\alpha(m)$ and $\beta(n)$ are as above. For any $A,\epsilon>0$,   we have
$$
\sum_{\substack{q\leq Q \\ (q,a)=1 }} \  \left| \sum_{\substack{ n\equiv a \pmod
q }} (\alpha*\beta)(n) - \frac {1}{\phi(q)} \sum_{\substack{ (n,q)=1 }}
(\alpha*\beta)(n) \right| \ll \| \alpha\| \| \beta\| \frac{x^{1/2}} {(\log x)^A}
$$
where $Q=x^{1-\epsilon}$, for every integer $a$.
\end{conjecture}

It is shown, in \cite{polymath8b} (by further developing the ideas described in this article), that if we assume this Generalized Elliott-Halberstam conjecture then for every admissible $3$-tuple of linear forms, there are infinitely many integers $n$ such that there are at least two primes amongst the three linear forms when evaluated at $n$. In particular, from the admissible forms $\{ x, x+2, x+6\}$ we deduce that there are infinitely many pairs of distinct primes that differ by no more than $6$.

The extraordinary work of Zhang breaks through the $\sqrt{x}$ barrier in some generality, working with moduli  slightly larger than $x^{1/2}$, though   his   moduli are $y$-smooth, with $y=x^\delta$.
The key result is as follows:

\begin{theorem} \label{BVdyadicrange}   Suppose that    $\alpha(m)$ and $\beta(n)$ are as above. There exist  constants $\eta, \delta>0$ such that, for any $A>0$,
$$
\sum_{\substack{q\leq Q \\  P(q)\leq x^\delta \\ (q,a)=1 \\ q \text{ squarefree}}} \ \left| \sum_{\substack{ n\leq x \\ n\equiv a \pmod
q }} (\alpha*\beta)(n) - \frac {1}{\phi(q)} \sum_{\substack{ n\leq x \\ (n,q)=1 }}
(\alpha*\beta)(n) \right| \ll_A \| \alpha\| \| \beta\| \frac{x^{1/2}} {(\log x)^A}
$$
where $Q=x^{1/2+\eta}$, for any integer $a$.
\end{theorem}

One then can deduce the same result with the support for $\alpha$ and $\beta$ expanded to the wider range
\[
x^{1/3}<m, n \leq x^{2/3} \text{  with  } mn\leq x.
\]
One proves this by carefully dissecting this range up into  into  dyadic ranges (that is, of the form $ M<m\leq 2M$  and $N<n\leq 2N$) as well as possible, and then carefully accounting for any $(m,n)$ pairs missed.

\subsection{The deduction of the main theorem  for primes} We will bound each term that arises from Vaughan's identity, \eqref{Vaughidentity}, with $U=V=x^{1/3}$, rewritten as
\begin{equation*}
\Lambda=\Lambda_{< x^{1/3}} + \mu_{<x^{1/3}} * L  - (\mu*\Lambda)_{< x^{1/3}}*1_{\geq x^{2/3}}  - \mu_{<x^{1/3}}*\Lambda_{< x^{1/3}}*1_{<x^{2/3}}  + \mu_{\geq x^{1/3}} * \Lambda_{\geq x^{1/3}} * 1.
\end{equation*}
 The first term is acceptably small, simply by taking absolute values. For the second term we write
$(\mu_{<x^{1/3}} * L)(n)= \sum_{um=n,\ u<x^{1/3}}  \mu(u)  \log m$, to obtain the difference
\[
\sum_{\substack{ u<x^{1/3} \\ (u,q)=1}} \mu(u) \left( \sum_{\substack{ x/u<m\leq 2x/u \\ m\equiv a/u \pmod q}}  \log m  - \frac 1{\phi(q)} \sum_{\substack{ x/u<m\leq 2x/u \\ (m,q)=1}}  \log m \right)
\]
Writing $M=x/u$, the inner sum is the difference between the sum of $\log m$ in $(M,2M]$ over an arithmetic progression  $b\pmod q$ with $(b,q)=1$, minus the average of such sums.
Now if $n_-=[M/q]$ and $n_+=[2M/q]$, then, since $\log q[m/q]<\log m <\log q([m/q]+1)$,
such a sum is $>\sum_{n_- \leq n\leq n_+-1} \log qn$ and is $<\sum_{n_-+1 \leq n\leq n_+1} \log qn$.
The difference between these bounds in $\ll \log M$, and hence this is our bound on the term in parentheses. Summing over $u$ yields a bound that is acceptably small.

We deal  with the third term, by the same argument as just above, since we obtain an inner sum of 1, over the values of $m$ in an interval of an arithmetic progression; and then we obtain a bound that is acceptably small.

We are   left to work with two sums of  convolutions:
\[
\sum_{\substack{mn\asymp x \\ mn\equiv a \pmod q}}   (\mu_{<x^{1/3}} * \Lambda_{< x^{1/3}})(m) 1_{<x^{2/3}}(n) \ \text{      and      } \
\sum_{\substack{mn\asymp x \\ mn\equiv a \pmod q}}  (\Lambda_{\geq x^{1/3}} * 1)(m) \mu_{\geq x^{1/3}}(n) ,
\]
where $x^{1/3} \ll m,n \ll x^{2/3}$, and each convolution takes the form $\alpha(m)\beta(n)$ with  $\alpha(m)$ and $\beta(n)$  as above. The result then follows from Zhang's result as discussed at the end of the last subsection.

\subsection{Further reductions}  We reduce Theorem \ref{BVdyadicrange} further. The first observation is that we can restrict our moduli to those with $<C\log\log x$ prime factors, for some large $C>0$, since the moduli with more prime factors are rare and thus contribute little to the sum. Since the moduli are $y$-smooth, they can be factored as
$qr$ where $N/(yx^\epsilon)< r \leq N/x^\epsilon$; moreover as   the modulus does not have a lot of prime factors, one can select $q$ and $r$ so that the smallest prime factor of $q$, denoted $p(q)$, is $\geq D_0:=x^{\epsilon/\log\log x}$. Hence we may also now assume

\begin{itemize}
\item $r\in (R,2R]$ with  $P(r)\leq y$ with  $y:=x^\delta$.
\item $q\in (Q,2Q]$ with $D_0<p(q)\leq P(q) \leq y$.
\item $N/(yx^\epsilon)< R \leq N/x^\epsilon$ and $x^{1/2}/(\log x)^B<QR\leq x^{1/2+\eta}$
\end{itemize}
In \cite{polymath8}, some gains are made by working instead with the full set of moduli that have this kind of convenient factorization, rather than restrict attention just to those moduli which are $y$-smooth.

We begin by noting that
\[
\sum_{ n\equiv a \pmod {qr} } \gamma(n) - \frac{1}{\phi(qr)} \sum_{\substack{ (n,qr)=1 }} \gamma(n) =
\]
\[
\sum_{ n\equiv a \pmod {qr} } \gamma(n) - \frac{1}{\phi(q)} \sum_{\substack{ (n,q)=1  \\  n\equiv a \pmod r }} \gamma(n)  +  \frac{1}{\phi(q)}  \left( \sum_{\substack{ (n,q)=1  \\  n\equiv a \pmod r }} \gamma(n)   - \frac{1}{\phi(r)} \sum_{\substack{ (n,q)=1 \\ (n,r)=1 }} \gamma(n)  \right)
\]
with $\gamma=\alpha*\beta$. We sum the absolute value of these terms,  over the moduli $d\in [D,2D]$, factored into $qr$ as above.  Since $\beta(n)$ satisfies the Siegel-Walfisz criterion, we can deduce that $\beta(n)1_{(n,q)=1}$ also satisfies it, and therefore Theorem \ref{BFI2} is applicable for $\alpha(m)\ast \beta(n)1_{(n,q)=1}$; this allows us to bound the sum of the second terms here, suitably.  Hence it remains to prove
 \begin{equation}\label{straw-2}
\sum_{\substack{q\in [Q,2Q]  \\ D_0<p(q)\leq P(q) \leq y}} \ \
\sum_{\substack{ r\in [R,2R], \\ P(r)\leq y\\ qr \ \text{squarefree}}}
\left|\sum_\stacksum{ n  \equiv a\pmod r}{n \equiv b\pmod q} (\alpha \ast \beta)(n) - \sum_\stacksum{ n \equiv a\pmod r}{ n  \equiv b'\pmod q}  (\alpha \ast \beta)(n) \right| \\
 \ll_A \| \alpha\| \| \beta\| \frac{x^{1/2}} {(\log x)^A},
\end{equation}
 for any integers $a,b,b'$ with $p(abb')> y$.

\section{Removing the weights, and an unweighted arithmetic progression}\label{prelim-red}

At first sight it seems unlikely that one can estimate \eqref{straw-2} since it  
involves unspecified sequences $\alpha$ and $\beta$. However we will use the Cauchy-Schwarz 
inequality to obtain an upper bound which depends only on the  mean squares of $\alpha$ and $\beta$.  

\subsection{Removing the weights} In this section we use Cauchy's inequality to ``unfold'' \eqref{straw-2}, 
so as to remove the $\alpha$ and $\beta$ from the innermost sum. Surprisingly, this simple procedure
can be applied so as to avoid much loss. 

In the left-hand side of \eqref{straw-2} we replace the absolute value in the $(q,r)$ term by a complex number $c_{q,r}$ of absolute value $1$,  to obtain, after a little re-arranging:
$$
\sum_r \sum_m \alpha(m)
\left(\sum_q \sum_{n:\ mn  \equiv  a\pmod r} c_{q,r} \beta(n) (1_{mn \equiv b\pmod q} - 1_{mn  \equiv  b'\pmod q})\right) .
$$
By the Cauchy-Schwarz inequality the square of this is
\[
\leq \sum_r \sum_m |\alpha(m)|^2 \leq  R \| \alpha \|^2
\]
times
\begin{equation}\label{sq}
\sum_r \sum_{m} \left|  \sum_q \sum_{n:\ mn  \equiv  a\pmod r}  c_{q,r} \beta(n) (1_{mn  \equiv  b\pmod q} - 1_{mn  \equiv  b'\pmod q})\right|^2 .
\end{equation}
When we expand the square, we obtain the sum of four terms of the form
\begin{align}
& \pm \sum_r \sum_{m}   \sum_{q_1,q_2} \sum_\stacksum{n_1,n_2}{ mn_1 \equiv mn_2 \equiv a\pmod r}
 c_{q_1,r} \overline{c_{q_2,r}  } \beta(n_1) \overline{\beta(n_2)} 1_{mn_1  \equiv b_1\pmod {q_1}} 1_{mn_2  \equiv b_2\pmod {q_2}} \notag \\
 = & \pm \sum_r   \sum_{q_1,q_2} \sum_\stacksum{n_1,n_2}{ n_1 \equiv n_2  \pmod r}
 c_{q_1,r} \overline{c_{q_2,r}  } \beta(n_1) \overline{\beta(n_2)}  \cdot
 \sum_{m} 1_{\substack{ m  \equiv b_1/n_1\pmod {q_1} \\  m   \equiv b_2/n_2\pmod {q_2} \\ m  \equiv a/n_1\pmod {r}}} \label{TheUnfolded}
\end{align}
where we get ``$+$'' when $b_1=b_2=b$ or $b'$, and ``$-$'' otherwise, since $(mn,qr)=1$.

We have achieved our goal of having an unweighted innermost sum. Indeed, if it is non-zero,\footnote{If the congruences are incompatible, then this sum cannot possibly contain any integers, and so is $0$. Since $(r,q_1q_2)=1$ they are compatible unless $b_1/n_1\equiv b_2/n_2 \pmod {(q_1,q_2)}$. Note that this criterion is irrelevant if $(q_1,q_2)=1$.}
 then it is just  the number of integers in an interval of an arithmetic progression with common difference $r[q_1,q_2]$.

\subsection{The main terms} The number of integers in an interval of length $M$, from an arithmetic progression with common difference $r[q_1,q_2]$ is
\[ \frac M{r[q_1,q_2]}+O(1).\]
We study now the sum of the ``main terms'', the $M/r[q_1,q_2]$. Firstly, for the terms with $(q_1,q_2)=1$ the main terms sum to
\[
\pm \sum_r   \sum_{\substack{q_1,q_2\\ (q_1,q_2)=1} } \sum_\stacksum{n_1,n_2}{ n_1 \equiv n_2 \pmod r}
 c_{q_1,r} \overline{c_{q_2,r}  } \beta(n_1) \overline{\beta(n_2)}  \cdot \frac{M}{rq_1q_2} ,
\]
which is independent of the values of $b_1,b_2$ and hence cancel, when we sum over the four terms (and the two `$+$', and two `$-$', signs).  For the terms with $(q_1,q_2)\ne 1$ we have
$(q_1,q_2)\ge D_0$ (since the prime factors of the $q_i$ are all $\geq D_0$), and it is not difficult to show that these are $\ll x(\log x)^{O(1)}/RD_0$, which is acceptably small.

\subsection{The error terms and the advent of exponential sums} The ``$O(1)$''s in \eqref{TheUnfolded} can add up to a total that is far too large. One can show that in most of the terms of the sum, the common difference of the arithmetic progression  is larger than the length of the interval, so the correct count is either $0$ or $1$: It is hardly surprising that an error term of
``$O(1)$'' is too insensitive to help us.

To proceed, instead of approximating, we will give a \emph{precise} formula for the number of integers in an arithmetic progression in an interval, using a sum of exponentials. By the Chinese Remainder Theorem, we can rewrite our triple of congruence conditions
\[
m  \equiv b_1/n_1\pmod {q_1}, \  m   \equiv b_2/n_2\pmod {q_2}, \ m  \equiv a/n_1\pmod {r}
\]
as one,
\[
m\equiv m_0(n_1,n_2) \pmod {q}
\]
where $q=rg\ell_1\ell_2$,
when there is a solution, which happens if and only if $b_1/n_1  \equiv b_2/n_2\pmod g$, where
  $g=(q_1,q_2)$ and we now define $\ell_1=q_1/g,\ \ell_2=q_2/g$.

To identify whether $m$ is in a given interval $I$, we use Fourier analysis. We define $e(t):=e^{2i\pi t}$ for any real number $t$, and then $e_q(t)=e(t/q)$.
The \emph{discrete Fourier transform} is defined by
\[
\hat f(h):=\sum_{b \pmod q} f(b)e_q(hb),
\]
for any function $f$ of period $q$.   If $f$ is any such function and $I(.)$ is the characteristic function for the interval $(M,2M]$, then
\begin{equation}\label{Plancherel}
\sum_{m\in I} f(m)   =  \frac 1q\   \sum_{h \pmod q}  \hat I(h) \hat f(-h) ,
\end{equation}
is an example of Plancherel's formula. This has a ``main term'' at $h=0$ (which is the same as the main term we found above, in that special case). The coefficients $\hat I(h)$ are easily evaluated and bounded:
\[
 \hat I(h) = \sum_{m=M+1}^{2M}  e_q(hm)= e_q(2hM) \cdot \frac{ e_q(hM)-1 }{ e_q(h)-1 } .
\]
The numerator has absolute value $\leq 2$ and, using the Taylor expansion, the denominator has absolute value
$\asymp |h|/q$. Hence
\[
|\hat I(h) | \ll \min \{ M, q/|h| \} ,
\]
We apply \eqref{Plancherel} with   $f=\sum_i c_i 1_{m\equiv a_i\pmod q}$, take absolute values, and use our bounds for $|\hat I(h) |$, to obtain
 \begin{equation}\label{ExponExpan}
\left|\sum_{i } c_i \left( \sum_{\substack{m\asymp M \\ m\equiv a_i \pmod q}} 1 - \frac{ M }{q} \right) \right| \ll  \sum_{\substack{0\leq j\leq J \\ H_j:=2^jq/M}}  \frac 1{H_j} \sum_{1\leq |h|\leq H_j}\left|\sum_{i} c_i e_q(a_i h)\right|  .
\end{equation}

The error terms in \eqref{TheUnfolded} are bounded  by
\[
\sum_{r\asymp R}   \sum_{g\leq G} \sum_ {\substack{\ell_1,\ell_2\asymp Q/g \\ (\ell_1,\ell_2)=1}}
\left|   \sum_{\substack{n_1,n_2\asymp N \\ n_1 \equiv n_2  \pmod r \\ b_1/n_1    \equiv b_2/n_2 \pmod g}} \beta(n_1) \overline{\beta(n_2)}  \cdot
\left(  \sum_{\substack{m\asymp M \\ m\equiv m_0(n_1,n_2) \pmod{rg\ell_1\ell_2}} } 1- \frac{M}{rg\ell_1\ell_2}
\right)
 \right|
 \]
which, by \eqref{ExponExpan},  is
\[
\ll \sum_{r\asymp R}   \sum_{g\leq G} \sum_ {\substack{\ell_1,\ell_2\asymp Q/g \\ (\ell_1,\ell_2)=1}}
\sum_{\substack{0\leq j\leq J \\ H_j:=2^jG/g}}  \frac 1{H_j} \sum_{1\leq |h|\leq H_j}
\left|   \sum_{\substack{n_1,n_2\asymp N \\ n_1 \equiv n_2  \pmod r \\  n_2 \equiv (b_2/b_1)n_1 \pmod g}} \beta(n_1) \overline{\beta(n_2)}  e_{rg\ell_1\ell_2}(m_0(n_1,n_2)h)  \right| .
 \]
 We write $n_1=n,\ n_2=n+kr$, replace the $n_2$ variable with $k$, and define $m_k(n)=m_0(n_1,n_2)$.
To simplify matters shall proceed with $r,g,k$ and $j$ fixed, and then sum over these at the end,
so we are reduced to studying
\begin{equation}\label{ToBeBounded}
\sum_ {\substack{\ell_1,\ell_2\asymp L \\ (\ell_1,\ell_2)=1}}
   \frac 1{H} \sum_{1\leq |h|\leq H}
\left|   \sum_{\substack{n\asymp N  \\  (b_2-b_1)n \equiv   b_1 kr   \pmod g}} \beta(n) \overline{\beta(n+kr)}  e_{rg\ell_1\ell_2}(m_k(n)h)  \right|
\end{equation}
where $L=Q/g$.

\section{Linnik's dispersion method}

The proof of Zhang's Theorem, and indeed of all the results in the
literature of this type, use Linnik's dispersion method.  The idea is to express the fact that $n$
belongs to an arithmetic progression using Fourier analysis; summing up over $n$ gives us a main term
plus a sum of exponential sums, and then the challenge is to bound each of these exponential sums.

Often
the sums come with weights, and judicious use of Cauchying allows one to work with an unweighted, but
more complicated exponential sum. We will discuss bounds on exponential sums later in this section. These exponential sums are often \emph{Kloosterman sums}, which one needs to bound. Individual Kloosterman sums can often by suitably bounded by Weil's  or Deligne's Theorem.
However, sometimes one needs to get good bounds on averages of Kloosterman sums, a question that was brilliantly attacked by Deshouillers and Iwaniec \cite{DI83}, using the (difficult) spectral theory
of automorphic forms. Indeed all previous work, breaking the $\sqrt{x}$ barrier, such as
\cite{fi-2}, \cite{bfi}) uses these types of estimates. One of the remarkable aspects of Zhang's work is that he avoids these penible techniques, and the restrictions that come with them.

Zhang was able to use only existing bounds on Kloosterman sums  to prove his Theorem, though he does use the sophisticated estimate of Birch and Bombieri from the appendix of \cite{fi-3}. Polymath8 indicates how even this deeper result can be avoided, so that the proof can be given using only ``standard'' estimates, which is what we do here.  In order to get the strongest known version one does need to involve these more difficult estimates, though they have been now reproved in an arguably more transparent way (see \cite{polymath8, Kowalski}).

\subsection{Removing the weights again}  To remove the $\beta$ weights from \eqref{ToBeBounded},
we begin by replacing the absolute value in \eqref{ToBeBounded} by the appropriate complex number $c_{h,\ell_1,\ell_2}$ of absolute value $1$, and re-organize to obtain
\begin{equation} \label{PreCauchy}
 \sum_{\substack{n\asymp N \\ (b_2-b_1)n\equiv b_1kr \pmod g}} \  \beta(n) \overline{\beta(n+kr)}
 \sum_{\substack{\ell_1,\ell_2\asymp L \\ (\ell_1,\ell_2)=1}}    \frac 1{H}  \sum_{1\leq |h|\leq H}
c_{h,\ell_1,\ell_2}  e_{rg\ell_1\ell_2}(m_k(n)h)  .
\end{equation}
We now Cauchy on the outer sum, which allows us to peel off the $\beta$'s in the term
\[
 \sum_{n}\  |\beta(n)  \beta(n+kr)|^2\leq   \sum_{n}\  |\beta(n) |^4=\| \beta\|_4^4,
\]
times the more interesting term
\[
 \sum_{n} \left|
 \sum_{\substack{\ell_1,\ell_2\asymp L \\ (\ell_1,\ell_2)=1}}  \frac 1{H}  \sum_{1\leq |h|\leq H}
c_{h,\ell_1,\ell_2}  e_{rg\ell_1\ell_2}(m_k(n)h)  \right|^2 .
\]
We simply expand this sum, and  take absolute values for each fixed $h,j,\ell_1,\ell_2, m_1, m_2$, to obtain
\[
\leq  \frac 1{H^2}   \sum_{1\leq |h|, |j| \leq H_i}  \sum_{\substack{\ell_1,\ell_2, m_1, m_2\asymp L \\ (\ell_1,\ell_2)=(m_1,m_2)=1}}
\left| \sum_{\substack{n\asymp N \\ (b_2-b_1)n\equiv b_1kr \pmod g}}    e_{rg\ell_1\ell_2}(m_k(n)h)
e_{rgm_1m_2}(-m_k(n)j)   \right| .
\]
Finally we have pure  exponential sums, albeit horribly complicated, having bounded the contributions of the unspecified sequences $\alpha$ and $\beta$ to our original sums by surprisingly simple functions.

Our approach here works for a rather more general class of (pairs of) sequences $\alpha$ and $\beta$. However, in order to appropriately bound $\| \beta\|_4^4$ in terms of $\| \beta\|_2^2$ for the theorem formulated here, it perhaps simplest to use the additional restriction on the size of each $\beta(n)$ that we gave here.

\subsection{Exponential sums with complicated moduli}  If $(r,s)=1$ then there are integers $a,b$ for which \[ ar+bs=1. \] Note that although there are infinitely many possibilities for the pair of integers $a,b$, the values of  $a \pmod s$  and $b\pmod r$ are uniquely defined. If we divide the previous equation by $rs$, and multiply by $m$, and then take $e(.)$ of both sides, we obtain
\[
e_{rs}(m) = e_s(am)\cdot e_r(bm) .
\]
This allows us to write the exponential, in our last sum, explicitly. After some analysis, we find that  the exponential sums at the end of the last subsection each take the form
\begin{equation} \label{ExpSum}
\sum_{\substack{n\asymp N \\ n \equiv a \pmod q}}    e_{d_1}\left( \frac{C_1}n \right)
e_{d_2}\left( \frac{C_2}{n+kr} \right) ,
\end{equation}
for some constants $C_1, C_2$ (where $d_1=rg[\ell_1,\ell_2],\ d_2= [m_1,m_2] $ and $q$ divides $g$). These constants depend on many variables but are independent of $n$. With a change of variable $n\mapsto qn+a$ we transform this to another sum of the same shape but now summing over all integers $n$ in a given interval.

\subsection{Exponential sums: From the incomplete to the complete} We now have  the sum of the exponential of a function of $n$, over the integers  in an interval. There are typically many integers in this sum, so this is unlike what we encountered earlier (when we were summing $1$).  The terms of the sum are periodic of period dividing $ [d_1,d_2]$ and it is not difficult to sum the terms over a complete period. Hence we can restrict our attention to ``incomplete sums'' where the sum does not include a complete period.

  We can now employ \eqref{Plancherel} once more. The coefficients $ \hat I(h) $ are well understood, but the $\hat f(h)$ now take the form
\[
\sum_{n \pmod q} e_{d_1}\left( \frac{C_1}{n} +hn \right)
e_{d_2}\left( \frac{C_2}{n+\Delta} +hn \right) ,
\]
a ``complete'' exponential sum.

The trick here is that we can factor the exponential into its prime factor exponentials and then, by the Chinese Remainder Theorem, this sum \emph{equals} the product over the primes $p$ dividing $q$,  of the same sum but now over $n\pmod p$ with the appropriate $e_p(\ast)$. Hence we have reduced this question to asking for good bounds on exponential sums of the form
\[
\sum_{n \pmod p} e_{p}\left( \frac{a}{n} + \frac{b}{n+\Delta} + cn \right) .
\]
Here we omit values of $n$ for which a denominator is $0$.
As long as this does not degenerate (for example, it would degenerate if $p|a,b,c$) then Weil's Theorem implies that this is $\leq \kappa p^{1/2}$, for some constant $\kappa>0$.  Therefore the complete sum over $n \pmod q$ is $\leq \kappa^{\nu(q)} q^{1/2}$. This in turn allows us to bound our incomplete sum
\eqref{ExpSum}, and to bound the term at the end of the previous section.

The calculations to put this into practice are onerous, and we shall omit these details here. At the end one finds that the bounds deduced are acceptably small if
\[
x^{1/2}\geq N>x^{(2+\epsilon)/5}
\]
where $\epsilon>12\eta+7\delta$. However this is not quite good enough, since we need to be able to take $N$ as small as $x^{1/3}$.

We can try a modification of this proof, the most successful being where, before we Cauchy equation \eqref{PreCauchy} we also fix the $\ell_1$ variable. This variant allows us to extend our range to all
\[
N>x^{\frac 13 +\epsilon }
\]
where $\epsilon>\frac{14}3 \eta + \frac 72 \delta$.  We are very close to the exponent $\frac 13$, but it seems that we are destined to just fail.

\section{Complete exponential sums: Combining information the Graham-Ringrose way}

The ``square-root cancellation'' for incomplete exponential sums of the form $ |\sum_n e_q(f(n))| $
for various moduli $q$, with the sum over $n$ in an interval of length $N<q$, is not quite good enough to obtain our results.

Graham and Ringrose \cite{graham}  proved that we can improve the (analogous) incomplete character sum bounds  when $q$ is smooth. Here we follow Polymath8 \cite{polymath8},\footnote{Who, in turn, essentially rediscovered an earlier argument of Heath-Brown (see section 9 of \cite{HBold}).} who showed how to modify the Graham-Ringrose argument to   incomplete exponential sums. This will allow us to reduce the size of $N$ in the above argument and prove our result.

\subsection{Formulating the improved incomplete exponential sum result}
For convenience we will write the entry of the exponential sum as $f(n)$, which should be thought of as taking the form $a/n+b/(n+\Delta)+cn$, though the argument is rather more general.  We assume that $N<q$, so that the Weil bound gives
\begin{equation} \label{weil1}
\left|\sum_n   e_q(f(n))\right| \ll   \tau(q)^A q^{1/2}   .
\end{equation}
for some constant $A$ which depends only on the degree of $f$.

In what follows we will assume that $q$ is factored as $q = q_1 q_2$, and we will deduce that
\begin{equation}\label{vdc-1}
 \left|  \sum_n e_q(f(n))\right| \ll
\left( q_1^{1/2}+      q_2^{1/4}  \right)    \tau(q)^A (\log q)  N^{1/2} .
\end{equation}
If $q$ is $y$-smooth then we let $q_1$ be the largest divisor of $q$ that  is $\leq (qy)^{1/3}$ so that it must be $>(q/y^2)^{1/3}$, and so $q_2\leq (qy)^{2/3}$. Hence the last bound implies
\[
 \left|  \sum_n  e_q(f(n))\right| \ll    \tau(q)^A   (qy)^{1/6} (\log q)  N^{1/2} .
\]
It is this bound that we insert into the  machinery of the previous section, and it allows use to extend our range to all
\[
N>x^{\frac 3{10}+\epsilon }
\]
where $\epsilon$ is bounded below by a (positive) linear combination of $\eta$ and $\delta$. In order that we can stretch the range down to \emph{all} $N>x^{\frac 13}$, this method requires that
\[
162 \eta +90 \delta <1.
\]

\subsection{Proof of \eqref{vdc-1}}   We may assume
\[  q_1 \leq N \leq q_2 \]
else  if $N<q_1$ we have the trivial bound $\leq N<(q_1N)^{1/2}$, and if $N>q_2$ then \eqref{weil1} implies the result since
$q^{1/2}=(q_1q_2)^{1/2}<(q_1N)^{1/2}$.

The main idea will be to reduce our incomplete exponential sum mod $q$, to a sum of incomplete exponential sums mod $q_2$. Now \[
e_q(f(n+kq_1)) = e_{q_1}( f(n)/q_2)\ e_{q_2}(  f(n+kq_1)/q_1 )
\]
so that, by a simple change of variable, we have
\[
\sum_n   e_q(f(n)) = \sum_n  e_q(f(n+kq_1)))  = \sum_n \  e_{q_1}( f(n)/q_2)\ e_{q_2}(  f(n+kq_1)/q_1 ).
\]
Now, if we sum this over all $k, 1\leq k\leq K:= \lfloor N/q_1\rfloor$, then we have
\[
K \sum_n   e_q(f(n)) = \sum_n e_{q_1}( f(n)/q_2)\ \sum_{k=1}^K    e_{q_2}(  f(n+kq_1)/q_1 ) ,
\]
and so
\begin{align*}
 \left| K \sum_n  e_q(f(n))\right|^2 &\leq \left( \sum_n \left|\sum_{k=1}^K  e_{q_2}(  f(n+kq_1)/q_1 )\right| \right)^2 \\
&\ll N \sum_n \left|\sum_{k=1}^K   e_{q_2}(  f(n+kq_1)/q_1 )\right|^2  \\
&  =N  \sum_{1 \leq k,k' \leq K} \sum_n   e_{q_2}( g_{k,k'}(n)) ,
\end{align*}
where $g_{k,k'}(n):=  (f(n+kq_1) - f(n+k'q_1) )/q_1 \pmod{q_2}$ if $n+kq_1,\ n+k'q_1\in I$, and $g_{k,k'}(n):=0$ otherwise.
If $k=k'$ then $g_{k,k}(n)=0$, and so these terms contribute $\leq KN^2$.

We now apply the bound of \eqref{weil1} taking $f=g_{k,k}$ for $k\ne k'$. Calculating the sum yields \eqref{vdc-1}.

\subsection{Better results} In \cite{polymath8} the authors obtain better results using somewhat deeper techniques.

By replacing the set of $y$-smooth integers by the much larger class of integers with  divisors in a pre-specified interval (and  such that those divisors have divisors in a different pre-specified interval, etc., since one can iterate the proof in the previous section)  they improve the restriction to
\[ 84 \eta + 48 \delta  < 1 .\]
Following Zhang they also gained bounds on certain higher order convolutions (of the shape $\alpha*1*1*1$), though here needing  deeper exponential sum estimates, and were then able to improve the restriction to (slightly better than)
\[ 43 \eta + 27 \delta  < 1 .\]

\subsection{Final remark} It is worth noting that one can obtain the same
quality of results only assuming a bound $\ll p^{2/3-\epsilon}$ for the relevant
exponential sums in finite fields.

\section{Further Applications} \label{FurtherApples}

Since Maynard's preprint appeared on the arxiv, there has been a flowering of diverse applications of the techniques, to all sorts of questions from classical analytic number theory. My favorite (in that it is such a surprising application) was given by Pollack \cite{pollack}, who connected these ideas to another famous problem:

$\bullet$\ In 1927 Artin conjectured that any integer $g$, that is not a perfect square and not $-1$, is a primitive root for infinitely many distinct primes. Following beautiful work of Gupta and Murty \cite{GuptaMurty}, Heath-Brown \cite{HB} showed that this must be true for all but, at most two, primes $g$. Hooley \cite{hooley} had shown how to prove Artin's conjecture assuming the Generalized Riemann Hypothesis. Pollack \cite{pollack} has now shown, assuming the Generalized Riemann Hypothesis, that any such integer $g$ is a primitive root for each of  infinitely $m$-tuples of primes which differ by no more than $B_m$.

Another beautiful result giving prime patterns:

$\bullet$\    Combining  ideas from this article with those from Green and Tao \cite{GT}, Pintz showed \cite{pintz-polignac} that there exists an integer $B>0$ such that are infinitely many arithmetic progressions of primes
 $p_n,\ldots, p_{n+k}$ such that each of $p_n+B,\ldots, p_{n+k}+B$ is also prime.

\subsection{Prime ideals} There are several analogous results in number fields.

$\bullet$\  Thorner \cite{thorner} showed that for any given binary quadratic form $f$ of discriminant $D<0$,
there are infinitely many pairs of distinct primes $p$ and $q=p+B(D)$ which are values of  $f$.

Let $A$ be  the ring of integers of a given number field $K$.

$\bullet$\ Thorner \cite{thorner}   showed that if $K/\mathbb Q$ is Galois then there are infinitely many pairs of prime ideals of $A$ whose norms are distinct primes that differ by $B(K)$.

$\bullet$\  Castillo, Hall, Lemke Oliver, Pollack and Thompson \cite{castillo} proved  that, for any admissible $k$-tuple $h_1,\ldots,h_k$ in $A$, there are infinitely many $\alpha\in A$ such that at least $m$ of the ideals $(\alpha+h_1),\ldots, (\alpha+h_k)$ are prime. Here we need $k$ to be sufficiently large as a function of $K$ and $m$.

\subsection{Applications to (irreducible) polynomials} Castillo, Hall, Lemke Oliver, Pollack and Thompson \cite{castillo} also proved  that, for any admissible $k$-tuple $h_1,\ldots,h_k$ of polynomials in $\mathbb F_q[t]$, there are infinitely many $f\in \mathbb F_q[t]$ such that at least $m$ of the polynomials $f+h_1,\ldots, f+h_k$ are irreducible in $\mathbb F_q[t]$. Here we need $k$ to be sufficiently large as a function of  $m$ (but not $q$); in fact, the same bound as was required in Maynard's result that was proved earlier.

\subsection{Quadratic twists of elliptic curves and coefficients of modular forms}
Thorner \cite{thorner} gave several applications to elliptic curves and modular forms:

$\bullet$\  For any given newform $f(z)=\sum_{n\geq 1} a_f(n)q^n$ for $\Gamma_0(N)$ of even weight $k\geq 2$, and any prime $\ell$ there exist infinitely many pairs of distinct primes $p$ and $q=p+B(f,\ell)$ for which
$a_f(p)\equiv a_f(q) \pmod \ell$.

$\bullet$\ There exist infinitely many pairs of distinct primes $p$ and $q=p+B_0$ such that the elliptic curves $py^2=x^3-x$ and $qy^2=x^3-x$ each have only finitely many rational points.

$\bullet$\  There exist infinitely many pairs of distinct primes $p$ and $q=p+B_1$ such that the elliptic curves $py^2=x^3-x$ and $qy^2=x^3-x$ each have infinitely many rational points.

These last two results can be generalized to other elliptic curves for which certain (technical) properties hold.

\subsection{Quantitative forms of the Maynard-Tao Theorem} These can give a lower bound for the number of $n\in [x,2x]$ for which at least $m$ of $n+a_1,\ldots,n+a_k$ are prime, and also can allow the $a_j$ to vary with $x$, getting as large as a multiple of $\log x$. Maynard \cite{clusters} generalized his result and proof, discussed earlier in this article, as follows:

$\bullet$\  Suppose that
we are given a finite admissible set of linear forms $\{ b_jx+a_j\}$, a sequence of  integers $\mathcal N$ which is ``well-distributed'' in arithmetic progressions on average, and a set of primes $\mathcal P$ such that each $\{ b_jn+a_j: n\in \mathcal N\} \cap \mathcal P$ is also ``well-distributed'' in arithmetic progressions on average. Then there are infinitely many integers $n\in \mathcal N$ such that at least $\ell$ of the $\{ b_jn+a_j: \ 1\leq j\leq k\}$ are primes in $\mathcal P$, where $\ell\gg \log k$. (This result can be used to deduce several of the others  listed here.)

$\bullet$\  One amazing consequence of  this result is that, for any $x,y\geq 1$, there are $\gg x\exp(-\sqrt{\log x})$ integers $n\in (x,2x]$ for which:

\centerline{\textsl{There are $\gg \log y$ primes in the interval $(n,n+y]$.}}

$\bullet$\   Let $d_n=p_{n+1}-p_n$ where $p_n$ is the $n$th smallest prime. Pintz  \cite{pintz-ratio} showed that
there are infinitely many $n$ for which $d_n,d_{n+k}$ are significantly larger than $\log n$ whereas each of
$d_{n+1},\ldots,d_{n+k-1}$ are bounded.

\subsection{The set of limit points, $\mathcal L$ , of the $\frac{p_{n+1}-p_n}{\log p_n}$ }
 It is conjectured that $\mathcal L=[0,\infty]$.  The result of GPY gave that $0\in L$, and it has long been known that $\infty\in L$ (see appendix \ref{largegaps}).

$\bullet$\  The quantitative form of the Maynard-Tao theorem immediately gives that if $B$ is a set of any 50 positive real numbers, then $(B-B)^+$ contains an element of $L$. From this one can deduce \cite{banks} that the measure of $L\cap[0,x]$ is at least $x/49$.

$\bullet$\  Pintz \cite{pintz-polignac} showed that $[0,c]\subset L$ for some $c>0$ (though the proof does not yield a value for $c$).

\appendix

\section{Hardy and Littlewood's heuristic for the twin prime conjecture} \label{HLheuristic}
The rather elegant and natural heuristic for the quantitative twin prime conjecture, which we described in section \ref{Primektuples}, was not the original way in which Hardy and Littlewood made this extraordinary prediction. The genesis of their technique lies in the \emph{circle method}, that they developed together with Ramanujan.  The idea is that one can distinguish the integer $0$ from all other integers, since
\begin{equation} \label{expintegral}
\int_0^1 e(nt) dt = \begin{cases} 1 &\text{ if } n=0;\\
0&\text{ otherwise, } \end{cases}
\end{equation}
where, for any real number $t$, we write $e(t) := e^{2\pi i t}$. Notice that this is literally an integral around the unit circle. Therefore to determine whether the two given primes $p$ and $q$ differ by $2$, we simply determine
\[
\int_0^1 e((p-q-2)t) \ dt.
\]
If we sum this up over all $p,q\leq x$, we find that the number of twin primes $p,p+2\leq x$ equals, exactly,
\[
\sum_{\substack{p,q \leq x \\ p,q \text{  primes}}} \int_0^1 e((p-q-2)t) \ dt =
 \int_0^1 |P(t)|^2 e(-2t) \  dt, \ \text{   where   } \ P(t):=   \sum_{\substack{p \leq x \\ p \text{  prime}}}  e(pt).
\]
In the circle method one next distinguishes  between those parts of the integral which are large (the \emph{major arcs}), and those that are small (the \emph{minor arcs}). Typically the major arcs are small arcs around those $t$ that are rationals with small denominators. Here the width of the arc is about $1/x$, and we wish to understand the contribution at $t=a/m$, where $(a,m)=1$.  We then have
\[
P(a/m) = \sum_{\substack{b  \pmod m \\ (b,m)=1}}  e_m(ab) \pi(x;m,b) .
\]
where $e_m(b)=e(\frac{b}{m}) = e^{2\pi i b/m}$. We note the easily proved identity
\[
\sum_{r  \pmod m,\ (r,m)=1} e_m(rk)= \phi((k,m)) \mu(m/(m,k)).
\]
Assuming the prime number theorem for arithmetic progressions with a  good error term we therefore see that
\[
P(a/m) \approx \frac x{\phi(m)\log x}  \sum_{\substack{b  \pmod m \\ (b,m)=1}}  e_m(ab) = \frac {\mu(m)}{\phi(m)}  \frac x{\log x} .
\]
Hence in total we predict that the number of prime pairs $p, p+2\leq x$ is roughly
\begin{align*}
&\approx \frac 1x \sum_{m\leq M} \sum_{a:\ (a,m)=1}  e_m(-2a) \left|  \frac {\mu(m)}{\phi(m)}  \frac x{\log x}  \right|^2 \\ &\approx
\frac x{(\log x)^2} \sum_{m\geq 1}  \frac {\mu(m)^2}{\phi(m)^2} \cdot \phi((2,m)) \mu(m/(2,m)) \\
&= \frac x{(\log x)^2} \left( 1 + \frac 1{\phi(2)} \right) \prod_{p>2}  \left( 1 - \frac 1{\phi(p)^2} \right)  =  C \frac x{(\log x)^2} ;
\end{align*}
the same prediction as we obtained in section \ref{Primektuples} by a very different heuristic. Moreover an analogous argument yields the more general  conjecture for prime pairs $p,p+h$.

Why doesn't  this argument lead to a proof of the twin prime conjecture? For the moment we have little idea how to show that the minor arcs contribute very little.  We know that the minor arcs can be quite large in absolute value, so to prove the twin prime conjecture we would have to find cancelation in the arguments of the integrand on the minor arcs.   Indeed it is an important open problem to find cancelation in the minor arcs in \emph{any} problem.

However, if we add more variables then appropriate modifications of this argument can be made to work.
Indeed it is this kind of circle method argument that   led to Helfgott's  recent proof \cite{HH} that every odd integer $\geq 3$ is the sum of no more than three primes.

\section{Stop the press! Large gaps between primes} \label{largegaps}
The average gap between primes $\leq x$ is about $\log x$. This article has focused on recent work to prove that there are many much smaller gaps. How about larger gaps? Can one prove that there are infinitely many gaps between consecutive primes that are much larger than $\log x$? In 1931, Westzynthius showed that for any constant $C>0$ there exist infinitely $n$ for which $p_{n+1}-p_n>C\log p_n$.
His idea is simply to find many consecutive integers each of which has a very small prime factor (so none of these integers can be a prime).  Erd\H os and Rankin developed this method improving the result to: There exists a  constant $C>0$ such that there are  infinitely $n$ for which
\[
p_{n+1}-p_n> C\log p_n \frac{\log\log p_n}{(\log\log\log p_n)^2} \log\log\log\log p_n.
\]
Subsequent papers increased the constant $C$, though were unable to show that one could take arbitrarily large $C$ (and Cramer conjectured that gaps can be much larger, even as large as $(\log p_n)^2$).  Erd\H os liked to stimulate research on his favourite questions by offering cash prizes. The largest prize that he ever offered was \$ 10,000, to whoever could show that one can take $C$, here,  to be arbitrarily large.

The GPY method was developed to prove that there are (far) smaller gaps between primes than the average.
It came as quite a surprise when, in August 2014, James Maynard \cite{mayn2} showed that one could ingeniously modify the argument for small prime gaps, to improve the Erd\H os-Rankin theorem for large prime gaps, not only getting arbitrarily large $C$ but replacing $C$ by something like $\log\log\log p_n$.  \textsl{The same week}, Ford, Green, Konyagin and Tao \cite{FGKT} modified the  Erd\H os-Rankin argument \textsl{very differently}, the main ingredient being a version of Green and Tao's \cite{GT} famous theorem on $k$-term arithmetic progressions of primes, to also show that one could take $C$ to be arbitrarily large.

It is an exciting time for gaps between primes.

\end{document}